\documentclass[10pt]{siamltex}
\usepackage{amsmath}
\usepackage{amsfonts}
\usepackage{amssymb}
\usepackage{graphicx}
\usepackage{color}
\usepackage{enumerate}
\usepackage{bm}
\usepackage{xcolor}
\usepackage{stmaryrd}
\usepackage{comment}

\newcommand{\be}{\begin{equation}}
\newcommand{\ee}{\end{equation}}
\newcommand{\bea}{\begin{eqnarray}}
\newcommand{\eea}{\end{eqnarray}}
\newcommand{\beas}{\begin{eqnarray*}}
\newcommand{\eeas}{\end{eqnarray*}}

\newcommand{\cU}{\mathcal{U}}

\DeclareMathOperator*{\argmin}{arg\,min}

\usepackage{soul} 
\usepackage[normalem]{ulem} 

\usepackage{caption}
\def\cE{{\cal E}}

\def\cJ{{\cal J}}

\def\cB{{\cal B}}

\def\cL{{\cal L}}
\def\cS{{\cal S}}

\def\E{{\mathbb E}}

\def\P{{\mathbb P}}
\def\R{{\mathbb R}}
\def\Z{{\mathbb Z}}

\def\bc{{\bm x}}
\def\bc{{\bm c}}

\def\bg{{\bm g}}
\def\bz{{\bm z}}
\def\by{{\bm y}}

\def\bA{{\bm A}}

\def\({\Bigl (}
\def\){\Bigr )}
\def\[{\Bigl [}
\def\]{\Bigr ]}
\def \<{\langle}
\def \>{\rangle}

\graphicspath{{./},{./figures/}}

\input{tcilatex}

\title{Analysis of sparse recovery for Legendre expansions using envelope bound}

\author{
Hoang~Tran\thanks{Department of Computational and Applied Mathematics, Oak Ridge National Laboratory, 1 Bethel Valley Road, P.O. Box 2008, Oak Ridge TN 37831-6164. email: \texttt{tranha@ornl.gov}. }
\and Clayton~Webster\thanks{Department of Mathematics, University of Tennessee, 1403 Circle Drive, Knoxville, TN 37916 and Department of Computational and Applied Mathematics, Oak Ridge National Laboratory, 1 Bethel Valley Road, P.O. Box 2008, Oak Ridge TN 37831-6164. email: \texttt{webstercg@ornl.gov}.}\
}

\begin{document}
\maketitle

\begin{abstract} 
We provide novel sufficient conditions for the uniform recovery of sparse Legendre expansions using $\ell_1$ minimization, where the sampling points are drawn according to orthogonalization (uniform) measure. So far, conditions of the form $m \gtrsim \Theta^2 s \times \textit{log factors}$ have been relied on to determine the minimum number of samples $m$ that guarantees successful reconstruction of $s$-sparse vectors when the measurement matrix is associated to an orthonormal system. However, in case of sparse Legendre expansions, the uniform bound $\Theta$ of Legendre systems is so high that these conditions are unable to provide meaningful guarantees. In this paper, we 
present an analysis which employs the envelop bound of all Legendre polynomials instead, and prove a new recovery guarantee for $s$-sparse Legendre expansions,  
$$
m \gtrsim {s^2} \times \textit{log factors}, 
$$
which is independent of $\Theta$. Arguably, this is the first recovery condition established for orthonormal systems without assuming the uniform boundedness of the sampling matrix. The key ingredient of our analysis is an extension of chaining arguments, recently developed in \cite{Bou14,ChkifaDexterTranWebster15}, to handle the envelope bound. Furthermore, our recovery condition is proved via restricted eigenvalue property, a less demanding replacement of restricted isometry property which is perfectly suited to the considered scenario. Along the way, we derive simple criteria to detect good sample sets. Our numerical tests show that sets of uniformly sampled points that meet these criteria will perform better recovery on average. 
\end{abstract}

\section{Introduction}
\label{sec:intro}
Compressed sensing (CS), originating as a signal processing technique \cite{CRT06,Donoho06}, has become an appealing approach for function approximation. Theory of CS indicates that if a function possesses a sparse representation in a known basis, it can be reconstructed from a limited number of suitable sampling points using nonlinear techniques such as convex optimization or greedy algorithms. In general, this number scales linearly with the sparsity level and only logarithmically with the size of representation, being much smaller than those required by traditional methods such as projection and interpolation. In recent years, there have been many works in the literature developing and analyzing CS approaches to the 
approximation of high-dimensional functions and parameterized PDE systems; we refer to 
\cite{MG12, YGX12, RS14, 
PHD14, HD15,JakemanNarayanZhou16,Brugiapaglia_MathComp18,Adcock15b,ChkifaDexterTranWebster15,DexterTranWebster17} and the reference therein. 

Let us consider a function $g: \cU\to  \mathbb{R}$ defined on the domain $\cU\subset \mathbb{R}^d$, endowed with a probability measure $\varrho$. We assume $g$ is well represented by the finite expansion
\begin{align}
\label{leg_exp}
g(\by) \simeq \sum_{{j} \in \cJ} c_{j} \Psi_{j}(\by),\ \ \by\in \cU,
\end{align}
where $\{\Psi_j\}_{j\in \cJ}$ is a pre-determined orthonormal system associated with $\varrho$ and indexed by a finite set $\cJ$. We define the corresponding polynomial subspace $\mathbb{P}_{\cJ} := \text{span}\{\Psi_j(\by): j\in \cJ\}$. In CS approach, one aims to recover $g$ by reconstructing unknowns $\{c_{j}\}_{{j}\in \cJ}$ from $m$ samples $g(\by_1),$ $\ldots, g(\by_m)$, which are drawn independently from the orthogonalization measure $\varrho$. Denote by $N$ the cardinality of $\cJ$, i.e., $N=\#(\cJ)$, $\bA$ the normalized sampling matrix and $\bg $ the normalized observation of the target function, i.e.,
\be
\label{defA}
\bm{A} := \(\frac{\Psi_j({\bm y}_i)}{\sqrt{m}}\)_{\substack{1\leq i\leq m \\ j\in \cJ}}  \in \R^{m\times N} ,\quad \text{and}\quad
\bg := \left(\frac{g({\bm y}_i)}{\sqrt{m}}\right)_{1\le i \le m} \in \mathbb{R}^m, 
\ee
this task amounts to solving for the coefficient vector $\bc = (c_{j})_{j\in \cJ}$ that satisfies $\bA  \bc = \bg$ (or, to be exact, $\|\bA  \bc - \bg\| \le \eta$ where $\eta$ is the truncation error accounting for the representation of $g$ by a finite series). Since the samples $g(\by_i)$ are costly to acquired in many applications, it is desirable to reconstruct $\bc$ with as small number of samples as possible, in particular, smaller than $N$. Although the system $\bA \bz = \bg$ becomes underdetermined, this is possible assuming that the expansion \eqref{leg_exp} is sparse, in which case $\bc$ can be reconstructed by several efficient algorithms, for instance, solving the convex optimization problem promoting sparsity
\begin{align}
\label{intro:l1}
 \argmin_{\bz\in \mathbb{R}^N} \|\bz\|_1\ \text{subject to}\ \bA \bz = \bg. 
\end{align}

The sparse recovery for the orthonormal expansions via $\ell_1$ minimization has shown to be very promising. It is supported by well-known theoretical estimates \cite{RudelsonVershynin08,FouRau13,Bou14,HR15,ChkifaDexterTranWebster15} that $s$-sparse expansions\footnote{expansions of the form \eqref{leg_exp} that have only $s$ nonzero coefficients} can be accurately recovered given the number of sampling points satisfying 
\begin{align}
\label{intro_cond}
m\gtrsim \Theta^2 s \times \text{log factors}, 
\end{align}
where $\Theta$ is the {uniform bound} of the underlying basis: 
$
\Theta = \sup_{{j}\in \cJ} \|{ \Psi}_{j}\|_{L^{\infty}({\mathcal U})}. 
$ 
For many orthonormal systems of interest such as trigonometric polynomials, Chebyshev polynomials, and Fourier transforms, $\Theta$ is a small constant, implying $m$ grows linearly with the sparsity level and at most logarithmically with the size of the expansion. These are indeed the favorable settings for sparse recovery. However, the case of Legendre expansions has proven to be problematic. There, the uniform bound $\Theta$ is large and $m$ determined as in \eqref{intro_cond} may be prohibitive. It is widely agreed that reconstructing directly from Legendre system and its orthogonalization (uniform) measure is not a good strategy; and several advanced techniques has been developed over the years to deal with this challenge, most notably, preconditioning approach and weighted $\ell_1$ minimization. The former method multiplies a suitable weight function to Legendre polynomials to form a preconditioned orthonormal system with significantly reduced $\Theta$, \cite{RauWard12,HD15,JakemanNarayanZhou16}, while the latter takes advantage of small uniform bounds of low-order Legendre polynomials to enhance the recovery of sparse vectors with some structures, \cite{RW15,Adcock15b,ChkifaDexterTranWebster15}.  

That being said, it can be argued that the condition \eqref{intro_cond} is too pessimistic for sparse Legendre expansions, and uniform bound should not be entirely relied on as the single tool to determine the number of samples. To illustrate this, let us conduct a simple test on the reconstruction of sparse Legendre polynomials in one-dimensional domain, $\cU = [-1,1]$, via $\ell_1$ minimization \eqref{intro:l1}. In this example, $\Psi_j \equiv L_j$ (the univariate Legendre polynomial of order $j$), $\bA$ is the sampling matrix associated with $L_j$ as in \eqref{defA}, where $\by_1,\ldots,\by_m$ are drawn according to uniform distribution in $\cU$. As $\|L_j\|_{L^{\infty}(\cU)} = \sqrt{2j+1}$ and $\Theta \ge \|L_{N-1}\|_{L^{\infty}(\cU)}$, condition \eqref{intro_cond} yields 
$$
m \gtrsim Ns \times \text{log factors},
$$ 
a trivial estimate since the required number of samples is now greater than $N$. In other words, there is no sound theory for the sparse recovery from underdetermined Legendre systems with uniform sampling points. Further, it is worth remarking that $m$ actually depends linearly on the maximal degree of the sparse expansions, which can only be greater than, if not equal to, $N$. To test the sharpness of condition \eqref{intro_cond}, we consider three different cases of $\cJ$: 
\begin{enumerate} [\ \ \ \ \  i.]
\item $\cJ = \{j\in \mathbb{N}: 1\le j \le 200\}$, 
\item $\cJ = \{j\in \mathbb{N}: 301\le j \le 500\}$,
\item $\cJ = \{j\in \mathbb{N}: 1801\le j \le 2000\}$.    
\end{enumerate}
In all these cases, the size of the finite Legendre expansions is fixed at 200; however, the maximal degree is varied from 200 (case i.), 500 (case ii.), to 2000 (case iii.). We aim to reconstruct sparse vectors $\bc$ of coefficients, with sparsity ranging from $5$ to $40$, given noiseless observations $\bg = \bA \bc$. For each sparsity level, we run $100$ trials, in each of which the support of $\bc$ is selected randomly and its coefficients are drawn from a Gaussian distribution. We fix the number of sample points $m = 100$ throughout the test, but generate different random set of samples $\{\by_i\}_{1\le i \le m}$ for each trial. The numerical results on averaged $\ell^2$ errors and successful recovery rates are presented in Figure \ref{fig:test1}. Two important observations can be made. First, sparse vectors with sparsity $15$ or less are accurately reconstructed, indicating that it is possible to recover sparse vectors directly from underdetermined Legendre systems. Second, the required number of samples should not depend on the maximal degree of Legendre expansions, evidenced by the fact that the errors and successful rates of sparse $\ell_1$ recovery are similar for three choices of $\cJ$ (even slightly worse for case i. where the maximum degree is smaller). This simple test suggests that there exists a gap between the performance of sparse Legendre recovery in practice and existing theory, in particular, condition \eqref{intro_cond}.     

\begin{center}
\begin{figure}[h]
\centering
\includegraphics[height=4.5cm]{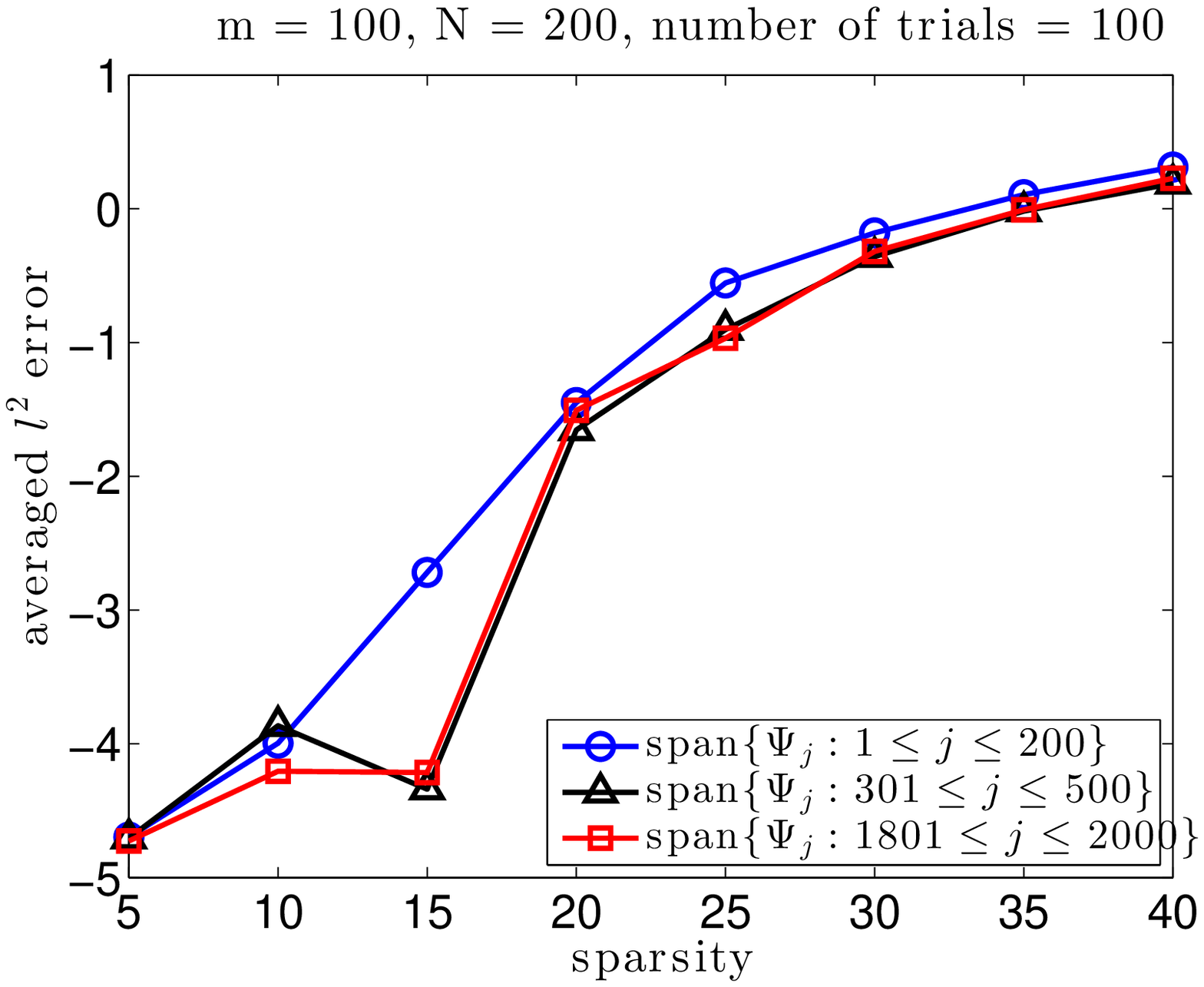} 
\includegraphics[height=4.5cm]{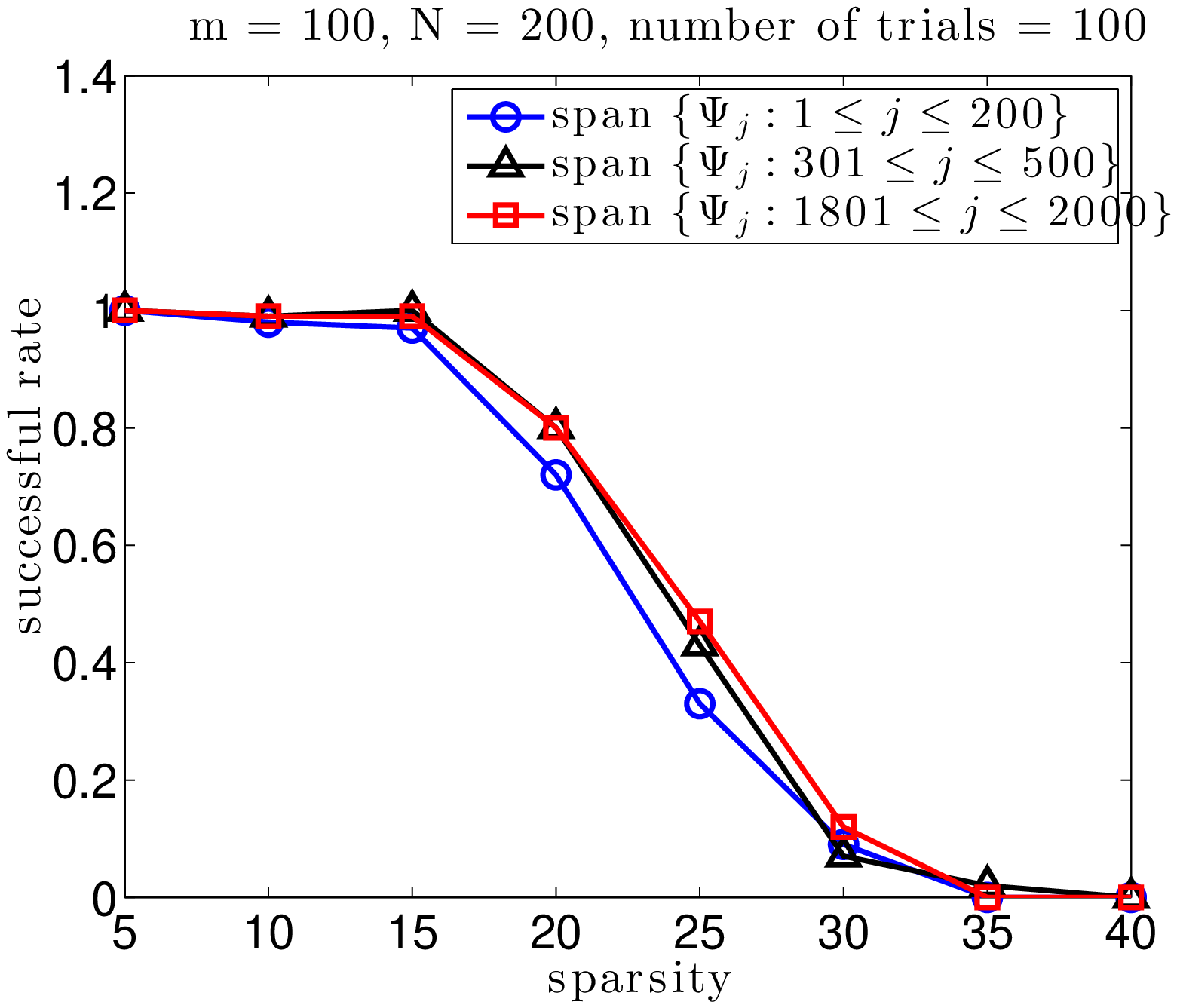} 
\caption{Sparse recovery for three Legendre expansions of the same size but different maximal degree.  }
\label{fig:test1}
\end{figure}
\end{center}

In this paper, we establish novel estimates on the sufficient number of samples for the sparse recovery of Legendre expansions using uniform sampling and $\ell_1$ minimization. In one-dimensional case, we prove that the uniform recovery of all $s$-sparse Legendre polynomials is guaranteed with high probability provided that 
\begin{align}
\label{cond:1d}
m\gtrsim s^2 \log^2(s) \log^2(N). 
\end{align}
We also prove the uniform recovery condition for $d$-dimensional case, i.e., $\cU = [-1,1]^d$, 
\begin{align}
\label{cond:multi-d}
\frac{m}{\log^{4d-4}(m) }\gtrsim C_d \, s^2 \log^2(s) \log^2(N), 
\end{align}
where $C_d$ is a universal constant depending only on the dimension. Note that here $N$ is the size of the polynomial expansion, which is not necessarily related to its maximal degree. The aforementioned estimates reveal that 
\begin{enumerate}
\item The number of samples grows quadratically in sparsity level, but only depends weakly on the size of expansions via log factors. Therefore, sparse recovery from {underdetermined} Legendre systems is provably possible, 
\item The number of samples is completely independent on the uniform bound $\Theta$ or maximal degree of Legendre expansions. 
\end{enumerate}
As such, they provide a mathematical justification for what we observed in the preliminary test. More importantly, these results bring about the first rigorous theory for $\ell_1$ recovery of Legendre approximations with uniform sampling. 

Our results are established via certain versions of null space property and incoherence condition of the sampling matrix. To guarantee successful reconstruction, it is sufficient for Legendre matrix to fulfill these conditions. Restricted isometry property (RIP) has typically been used to obtain the uniform recovery. Herein, however, we do not assume a fixed, known maximal degree for polynomial subspace $\mathbb{P}_{\cJ}$.  
Without this boundedness, we suspect that it is very challenging, if possible at all, to prove the eigenvalues of ``restricted'' matrices are bounded from above as required by RIP. To overcome this difficulty, we instead establish the uniform recovery via a variant of restricted eigenvalue property \cite{Bickel-Ritov-Tsybakov09,Koltchinskii2009a,Koltchinskii2009b}, which is a related condition but less demanding than RIP. In fact, restricted eigenvalue property only assumes some kind of positive definiteness of the sampling matrix, and many important matrices, such as correlated Gaussian and sub-Gaussian designs, have been shown to satisfy the restricted eigenvalue property while violate the RIP \cite{Buhlmann-deGeer09,Raskutti-Wainwright-Yu10}. As we shall see, null space property can be implied directly from restricted eigenvalue property.
A major part of our paper therefore is devoted to deriving precise bounds on the number of samples such that the restricted eigenvalue property is fulfilled.   

For this goal, our analysis employs \textit{chaining} techniques, which provide probabilistic bounds for stochastic processes via careful constructions of progressively finer nets upon which union bounds can be applied in an efficient manner. This approach has been used to establish the RIP estimates of several bounded orthonormal systems \cite{Bou14,HR15,ChkifaDexterTranWebster15}. However, unlike all previous works on the topic, we do not rely on the uniform bound, but rather apply and combine the envelope (i.e., pointwise) bound of Legendre systems to the chaining process. As illustrated in Figure \ref{fig:example2}, envelope bound is more efficient than uniform bound and invariant to the selection of polynomial subspaces. Thus, our new arguments using envelope bound lead to the improved estimates \eqref{cond:1d}-\eqref{cond:multi-d}, which do not depend on the polynomial subspace. Independently of the main results, we believe that this approach should be of interest on its own, and can be extended to other random matrices which can be bounded precisely pointwise, but whose uniform bound is bad. As a by-product, we derive some simple criteria to predict good (preferable) sample sets. We conduct numerical tests showing that sets of uniformly sampled points that meet these criteria will perform better recovery on average.

Our paper is organized as follows. In Section \ref{sec:WREP}, we review the restricted eigenvalue property and its relation to null space property and exact recovery guarantees. In Section \ref{sec:1d_anal}, we prove a new bound for restricted eigenvalue property of subsampled one-dimensional Legendre matrices. The restricted eigenvalue property estimate for multi-dimensional Legendre matrices will be discussed in Section \ref{sec:mult}. 
Section \ref{sec:num_examples} provides numerical illustrations on our criteria to select sample sets. Concluding remarks are given in Section \ref{sec:conclusion}. For the rest of this paper, $C$ represents a universal constant, whose value may change from place to place but is independent of any parameter.   

\begin{center}
\begin{figure}[h]
\centering
\includegraphics[height=4.5cm]{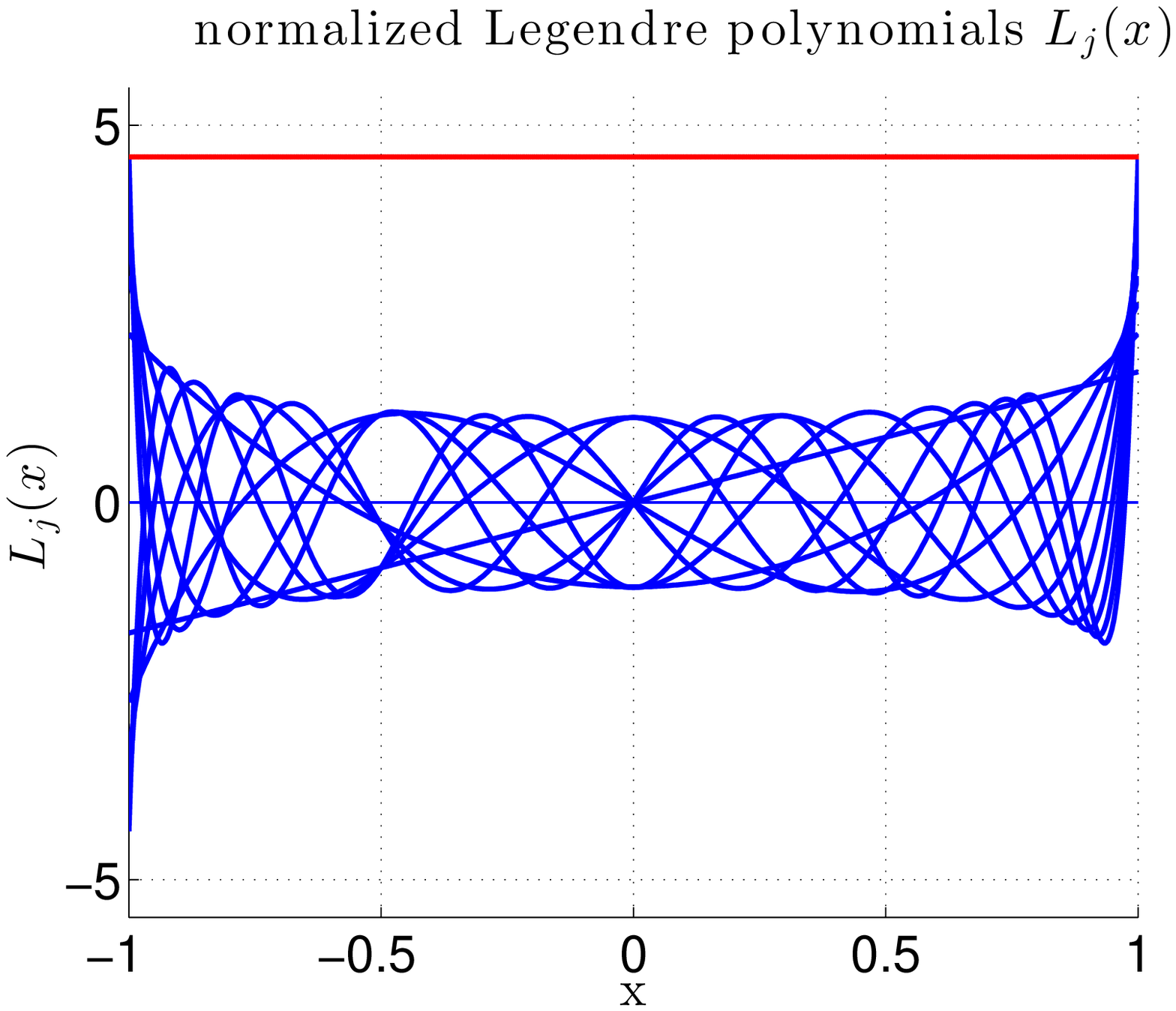} 
\includegraphics[height=4.5cm]{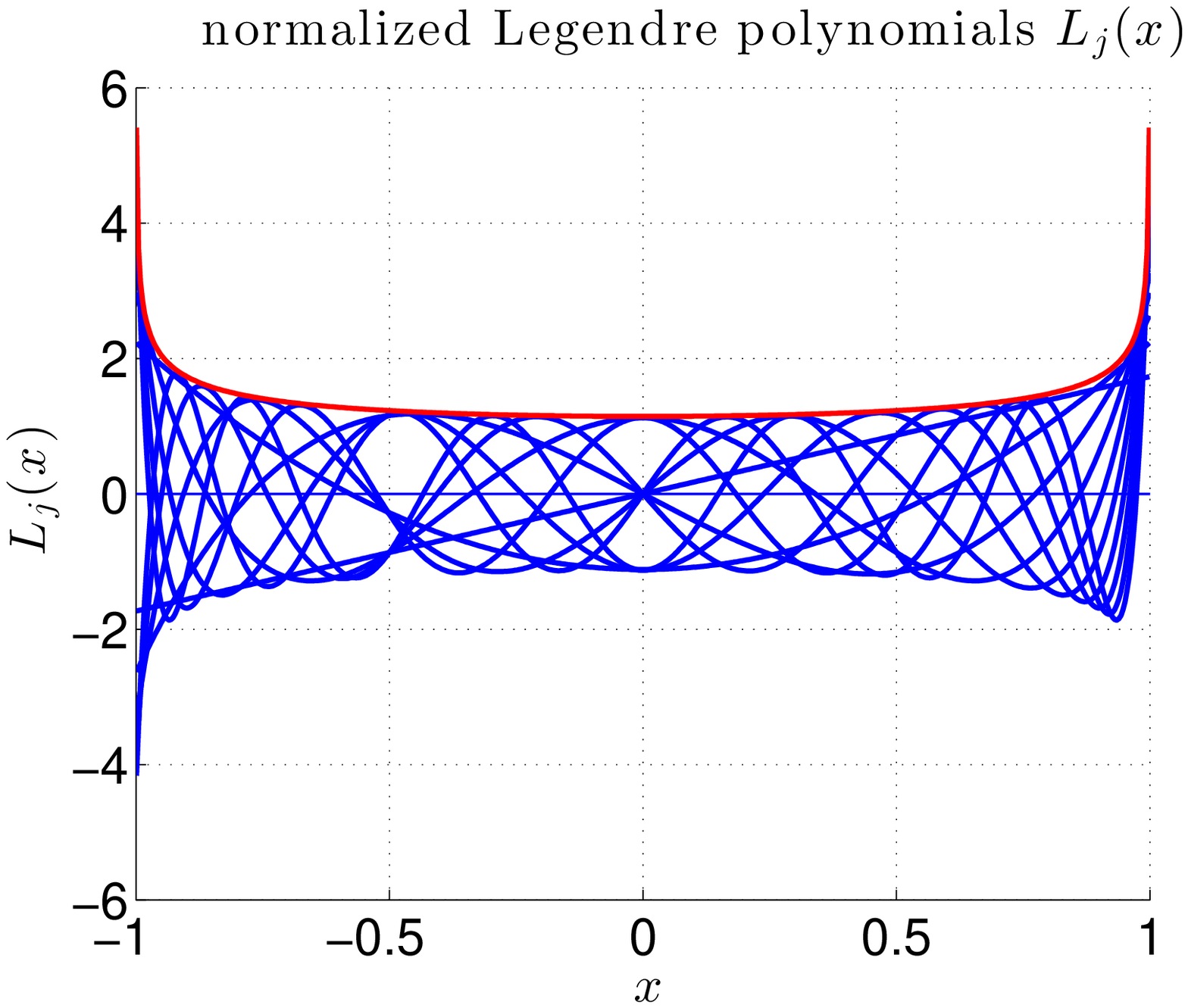} 
\caption{Uniform bound (left) and envelope bound (right) for one-dimensional Legendre system. }
\label{fig:example2}
\end{figure}
\end{center}

\section{Uniform recovery guarantee via restricted eigenvalue property} 
\label{sec:WREP}

In this section, we present an overview of restricted eigenvalue property and the implication of null space property and uniform recovery from restricted eigenvalue property. The discussion in this section does not assume $\bA$ being a subsampled Legendre system, but is also applicable for generic sampling matrix. First, the definition of restricted eigenvalue property is provided below. 

Let $\bz \in \mathbb{R}^N$ and be indexed in $\cJ$. For $S \subset \cJ$, we denote $S^c$ the complement of $S$ in $\cJ$ and $\bz_S$ the restriction of $\bz$ to $S$. For $\alpha>1$ and $0< s\le N$, define
\begin{gather}
\label{def:sparse_set}
\begin{aligned}
&{C}(S;\alpha) := \left\{\bz\in \mathbb{R}^N:  \|\bz_{S^c}\|_1 \le {\alpha}{\sqrt{s}} \|\bz_S\|_2\right\}, 
\\
 \text{ and } \ &{C}(s;\alpha) :=  \bigcup_{\substack {S\subset \cJ,\ \#(S) = s}} {C}(S;\alpha). 
\end{aligned}
\end{gather}
Let $B_1$ and $\partial B_1$ be the unit $\ell_1$-ball and sphere in $\mathbb{R}^N$ correspondingly
$$
B_1 = \{\bz \in \mathbb{R}^N: \|\bz\|_1 \le 1 \}, 
\text{\ \  and \ \ }
\partial B_1 = \{\bz \in \mathbb{R}^N: \|\bz\|_1 = 1 \}. 
$$ 
Similarly, denote by $B_2$ and $\partial B_2$ the unit $\ell_2$-ball and sphere in $\mathbb{R}^N$. 

\begin{definition}[Restricted eigenvalue property]
\label{def:WREP}
Given $\bA \in \mathbb{R}^{m\times N}$, $\alpha>1$ and $0< \delta < 1$, we say $\bA$ satisfies the restricted eigenvalue property of order $s$ with parameters $(\alpha,\delta)$ if 
\begin{align}
&  \|\bA\bz\|_2^2 \ge  (1 - \delta)\|\bz\|_2^2,  \ \ &\text{ for all }\ \bz \in {C}(s;\alpha),  \notag
\\
\text{ or equivalently, }& \|\bA\bz\|_2^2 \ge  1 - \delta,  \ \ &\text{ for all }\ \bz \in \partial B_2 \cap {C}(s;\alpha). 
\label{def:WREP2}
\end{align}
\end{definition}

We note that the restricted eigenvalue property stated above are slightly different from those often found in the literature, e.g., \cite{Bickel-Ritov-Tsybakov09,Buhlmann-deGeer09,Raskutti-Wainwright-Yu10}. In particular, the inequality $\|\bz_S\|_1 \le  {\sqrt{s}} \|\bz_S\|_2 $ yields that the set $C(s;\alpha)$ defined in \eqref{def:sparse_set} are slightly larger than 
$$
\bigcup_{\substack {S\subset \cJ,\ \#(S) = s}}  \left\{\bz\in \mathbb{R}^N:  \|\bz_{S^c}\|_1 \le {\alpha} \|\bz_S\|_1\right\}
$$
as considered in those works. 

The theory of restricted eigenvalue property, as a sufficient condition for the success of $\ell_1$ minimization methods, is well-developed in the fields of statistical inference and machine learning, where the RIP was proven to be severe for many random matrices of interest, for example, random designs with substantial correlation among rows/columns. In this paper, we develop recovery guarantees for a class of random matrices associated with Legendre orthonormal systems. As we do not assume the uniform boundedness, the fact that restricted eigenvalue property does not require an upper bound of $\|\bA \bz\|_2$ is also extremely useful in our analysis.  

Next, we will present the implication of null space property and uniform recovery from restricted eigenvalue property. We also include discussions on estimates for restricted eigenvalue property of orthonormal systems in case uniform bound is used. These derivations are quite simple and based on the corresponding developments for the RIP. However, as far as we know, they are not very well-spread in the compressed sensing literature.  

\begin{definition}[Restricted isometry property]
\label{def:RIP}
Given a matrix $\bA\in \mathbb{R}^{m\times N}$, we say $\bA$ satisfies the restricted isometry property of order $s$ if there exists $0 < \delta < 1$ such that 
\begin{align}
& (1 - \delta)\|\bz\|_2^2 \le \|\bA\bz\|_2^2 \le (1 + \delta)\|\bz\|_2^2,\   \forall \bz \in \mathbb{R}^N \text{ satisfying } \#(\supp(\bz))\le s,  \notag
\\
\text{ or} & \text{ equivalently, \ \ } 1 - \delta \le \|\bA\bz\|_2^2 \le 1 + \delta,\   \forall \bz \in \mathcal{E}_s, \label{def:RIP2}
\end{align}
where $\mathcal{E}_s := \{\bz\in \mathbb{R}^N: \#(\supp(\bz))\le s,\, \|\bz\|_2 = 1 \}$. 
\end{definition}

Compared to restricted eigenvalue property, RIP needs $\|\bA\bz\|_2$ to be bounded from both above and below, but only on the set of exactly sparse vectors in $\mathbb{R}^N$. In particular, it is easy to see that $\mathcal{E}_{2s}$ is a strict subset of $\partial B_2 \cap C(s;\alpha)$ required in \eqref{def:WREP2}. However, in available RIP estimates, \eqref{def:RIP2} is typically established for a larger set, that is, $\sqrt{s} B_1$, the $\ell_1$-ball that contains $\mathcal{E}_s$, \cite{Candes-Tao06,RudelsonVershynin08,Rauhut10,Bou14,ChkifaDexterTranWebster15}. 

Observe that $\partial B_2 \cap C(s;\alpha)$ is contained in the $\ell_1$-ball of radius $(1+\alpha)\sqrt{s}$. 
\begin{lemma}
\label{lemma:subset}
For $\alpha >1$, $0\le s\le N$, there follows
$ \partial B_2 \cap C(s;\alpha)  \subset (1 + \alpha) \sqrt{s} B_{1}. 
$
\end{lemma}
\begin{proof}
Let $\bz \in \partial B_2 \cap C(s;\alpha) $. Then, $ \|\bz_{S^c}\|_1 \le {\alpha}{\sqrt{s}} \|\bz_S\|_2$ for some $S\subset \cJ$, $\#(S) = s$. We have 
$
\|\bz\|_1 = \|\bz_S\|_1 + \|\bz_{S^c}\|_1 \le    {\sqrt{s}} \|\bz_S\|_2 + {\alpha}{\sqrt{s}} \|\bz_S\|_2 = (1 + \alpha) \sqrt{s} \|\bz\|_2,
$
therefore, $\bz \in  (1 + \alpha) \sqrt{s} B_{1}$. 
\end{proof}

One can easily derive from available RIP proofs that for orthonormal systems with uniform bound $\Theta$, RIP holds over the extended set $(1 + \alpha) \sqrt{s} B_{1}$ under a modified number of samples where the sparsity $s$ is replaced by $(1+\alpha)^2 s$. This property automatically implies the restricted eigenvalue property as a consequence of Lemma \ref{lemma:subset} and therefore, the restricted eigenvalue property holds under the same sample complexity. 

For a specific example, we can obtain the following result by slightly modifying \cite[Theorem 2.2]{ChkifaDexterTranWebster15}. 

\begin{proposition} [Restricted eigenvalue estimate for bounded orthonormal systems]
\label{note:RIP_theorem}
Let $\delta, \gamma$ be fixed parameters with 
$0<\delta < 1$, $0 < \gamma < 1$ and $\{\Psi_{j}\}_{{j}\in \cJ}$ be an orthonormal 
system of finite size $N = \#(\cJ)$ and $
\Theta = \sup_{{j}\in \cJ} \|{ \Psi}_{j}\|_{L^{\infty}}. 
$ Let $s_\alpha = (1+\alpha)^2 s$. Assume that 
\begin{align*}
m \ge C~\frac{\Theta^2 s_\alpha}{\delta^2} \log\( \frac{\Theta^2 {s_\alpha}}{\delta^2}\)
\max \biggl\{ \frac{2^5}{\delta^4}  \log\( 40 \frac{ \Theta^{2} s_\alpha}{\delta^2} 
\log\( \frac{\Theta^2 {s_\alpha}}{\delta^2}\) \) & \log(4N) , 
\\
\frac 1 \delta\log& \(\frac{1}{\gamma \delta }  
\log\( \frac{\Theta^2 {s_\alpha}}{\delta^2}\)  \) \biggl\} , \notag
\end{align*}
and $\by_1,\by_2,\ldots,\by_m$ are drawn independently from the orthogonalization 
measure $\varrho$ associated to $\{\Psi_{j}\}$. Then with probability exceeding $1-\gamma$, 
the normalized sampling matrix $\bA \in \mathbb{R}^{m\times {N}}$ (defined as in \eqref{defA}) satisfies
\begin{align*}
 \|\bm{Az}\|_2^2 > (1 - \delta) \|\bz\|^2_2 , \ \ \forall \bz \in {C}(s;\alpha).
\end{align*}
\end{proposition}


Interestingly, restricted eigenvalue property of order $s$ straightforwardly leads to null space property of order $s$.\footnote{Therefore, restricted eigenvalue property of order $s$ is comparable to RIP of order $2s$} 

\begin{proposition}
\label{lemma:NSP}
Assume that $\bA\in \mathbb{R}^{m\times N}$ satisfies the restricted eigenvalue property of order $s$ with parameters $(\alpha,\delta)$. Then, $\bA$ satisfies the null space property of order $s$, that is, for all $\bz \in \mathbb{R}^N$ and all $S\subset \cJ$ with $\#(S) = s$, 
\begin{align}
\label{NSP}
\|\bz_S\|_2 \le \frac{\rho}{\sqrt{s}}\|\bz_{S^c}\|_1 + \tau\|\bA\bz\|_2, \ \ \text{where}\ \rho = \dfrac{1}{\alpha}\text{ and }\tau = \dfrac{1}{ \sqrt{1 - \delta} }.
\end{align}
\end{proposition}
\begin{proof}
For any $\bz\in \mathbb{R}^N$, we only need to prove \eqref{NSP} for $S$ chosen as the index set corresponding to $s$ largest components of $\bz$. Consider two cases:

If $\bz \in {C}(S;\alpha)$, by \eqref{def:WREP2}, we have 
$$
\|\bz_S\|_2 \le \|\bz\|_2 \le \frac{1}{ \sqrt{1 - \delta} }\|\bA\bz\|_2 \le  \frac{\rho}{\sqrt{s}}\|\bz_{S^c}\|_1 + \tau\|\bA\bz\|_2  . 
$$

Otherwise, $\bz \notin {C}(S;\alpha)$ yields 
$$
\|\bz_S\|_2 \le  \frac{1}{\alpha \sqrt{s}}\|\bz_{S^c}\|_1  \le   \frac{\rho}{\sqrt{s}}\|\bz_{S^c}\|_1 + \tau\|\bA\bz\|_2. 
$$
\end{proof}

Recall $\bc$ and $\bg$ are the vectors of exact coefficients and observations such that $\|\bA  \bc - \bg\|_2 \le \eta$. Let $\bc^\#$ be an approximation of $\bc$ determined by  
\begin{align}
\label{intro:l1_noisy}
\bc^\# =  \argmin_{\bz\in \mathbb{R}^N} \|\bz\|_1\ \text{subject to}\ \|\bA \bz - \bg\|_2 \le \eta. 
\end{align}
Denote by $\sigma_s(\bc)_1$ the error of best $s$-term approximation of $\bc$ in $\ell_1$ norm, i.e., 
$
\sigma_s(\bc)_1  = \inf\limits_{\Lambda \subset \cJ,\, \#(\Lambda) \le s} \|\bc -\bc_{\Lambda}\|_1. 
 $
 The reconstruction error estimates follow from Proposition \ref{lemma:NSP} using standard analysis, e.g., \cite[Theorems 4.19 and 4.22]{FouRau13}.
 \begin{proposition}
 \label{prop:error_est}
Let $\bA \in \mathbb{R}^{m\times N}$, $\alpha>1$, $0< \delta < 1$ and $\delta_0 = \dfrac{1}{1+\alpha} < \dfrac{1}{2}$. Assume that $\bA$ satisfy \eqref{NSP}  for all $\bz \in \mathbb{R}^N$ and all $S\subset \cJ$ with $\#(S) = s$. For $\bc \in \mathbb{R}^N$ and $\bg \in \mathbb{R}^m$ satisfying $\|\bA  \bc - \bg\|_2 \le \eta$, let $\bc^\#$ be the solution of \eqref{intro:l1_noisy}. Then
 \begin{align}
\label{error_const:WREP}
\|\bc - \bc^{\#}\|_1 \le \frac{2}{1-2\delta_0} \sigma_s(\bc)_1 + \frac{4}{\sqrt{1-\delta}} \cdot \frac{1-\delta_0}{1-2\delta_0}  \eta {\sqrt{s}}. 
\end{align} 
\end{proposition}

In combining Propositions \ref{note:RIP_theorem}--\ref{prop:error_est}, for bounded orthonormal systems, the robust and stable recovery \eqref{error_const:WREP} is guaranteed with high probability given the number of samples
\begin{align}
\label{remark:WREP_con1}
m \gtrsim \delta_0^{-2} C_{\delta} \Theta^2 s\  \times \mbox{ log factors}, 
\end{align}
where $C_\delta$ represents the dependence of $m$ on $\delta$. As we derive Proposition \ref{note:RIP_theorem} from RIP estimates in \cite{Bou14,ChkifaDexterTranWebster15}, $C_\delta = \delta^{-6}$. However, better dependence on $\delta$ is possible, for example, $C_\delta = \delta^{-2}$ if applying the RIP estimates in \cite{Rauhut10,HR15}. Also, the log factor terms can take several different forms. Generally, large $\delta$ and $\delta_0$ are desirable for smaller number of samples, but this also lead to worse constants in reconstruction error estimates.

At first glance, the above number of samples established via restricted eigenvalue property seems to be more demanding than that obtained from the RIP ($m \gtrsim C_{\delta} \Theta^2 s\  \times ${ log factors}) in the sense that besides $\delta$, $m$ additionally depends on $\delta_0$. However, we can see that the dependence of error estimate constants on $\delta$ is different for these two approaches. In RIP-based error estimates, $\delta$ needs to be strictly bounded away from $1$ and the error grows quite fast with moderately large $\delta$, for example, with $\delta < 4/\sqrt{41}$, one has 
\begin{align*}
 \|\bc - \bc^{\#}\|_1 \le \frac{2({\sqrt{1-\delta^2}  + 3\delta/4})}{ {\sqrt{1-\delta^2} - 5\delta/4}} \sigma_s(\bc)_1 +  \frac{4 \sqrt{1+\delta}}{ {\sqrt{1-\delta^2} - 5\delta/4}}\cdot   \eta {\sqrt{s}},
\end{align*}
see \cite[Theorem 6.12]{FouRau13}; hence, an optimal dependence of $m$ on $\delta$ is critical. For the restricted eigenvalue property, as shown in \eqref{error_const:WREP}, $\delta$ appears only in the noise term, mildly affects the errors via the multiplication factor $1/\sqrt{1-\delta}$ and can be chosen close to $1$ without increasing this term much. As such, $C_\delta$ becomes less important. A more profound penalty arises from $\delta_0$, which more or less plays the role of $\delta$ in RIP approach. However, from \eqref{remark:WREP_con1}, we have seen that the number of measurements $m$ grows at rate $\delta_0^{-2}$ as $\delta_0$ decreases. This agrees with the best known dependence of $m$ on $\delta$ using RIP.




As uniform recovery follows directly from restricted eigenvalue property (Propositions \ref{lemma:NSP}--\ref{prop:error_est}), the rest of this paper is devoted to derive precise bounds on the number of samples such that the restricted eigenvalue property holds for random Legendre systems. The advantage of restricted eigenvalue property in refraining us from the upper bound of $\|\bA \bz\|_2$ is not clear in the above discussion, where the uniform bound of orthonormal systems is available. In the next part, we do not assume this information and exploit this advantage to its full strength. 

\section{Restricted eigenvalue property for univariate Legendre systems}
\label{sec:1d_anal}
We present an estimate for restricted eigenvalue property for one-dimensional Legendre basis. In this section, $\cU := [-1,1]$, $\varrho$ is the uniform probability measure on $\cU$, $\cJ$ is a subset of $\mathbb{N}$, and $\{\Psi_{j}\}_{{j}\in \cJ} \equiv \{L_{j}\}_{{j}\in \cJ}$, the system of univariate Legendre polynomials orthonormal with respect to $\varrho$. We denote
\begin{align*}
\psi(\by,\bz) &:= \sum_{j \in \cJ} {z}_{j} L_{j}(\by)  ,\quad\quad \by\in \mathcal{U},\;\; \bz \in \R^N.
\end{align*}
The envelope bound of Legendre polynomials is central to our analysis. It is well known that 
\begin{align}
\label{envelope}
|L_{j}(\by)| \le \frac{2/\sqrt{\pi}}{\sqrt[4]{1-\by^2}}, \ \forall \by \in (-1,1),\ j \in \mathbb{N}, 
\end{align}
see \cite{AntonovHolshevnikov81,Adcock15b}. This estimate is sharp in the sense that the constant $2/\sqrt{\pi}$ cannot be improved (see also an illustration in Figure \ref{fig:example2}). Let us define 
$
\Omega (\by) := \dfrac{2/\sqrt{\pi}}{\sqrt[4]{1-\by^2}}. 
$

We introduce our notion of ``good sample sets'', which we refer to as {preferable sets} throughout.

\begin{definition}
\label{def:PSS}
For $m > 0$, $0 < \gamma_0 < 1$, let $Q = \{\by_1,\ldots,\by_m\}$ be a set of samples where $\by_1,\ldots,\by_m$ are drawn independently from the uniform distribution on $[-1,1]$. Define $m$ i.i.d random variables 
\begin{align}
\label{define:z}
Z_i =  { \Omega(\by_i)} \exp\(-\dfrac{1}{2 \Omega^2(\by_i)}  \sqrt{\dfrac{m}{\gamma_0}} \),\quad \forall i \in \{1,\ldots, m\}. 
\end{align}
We name $T(Q)= \sum_{i=1}^m Z_i  $ the \textbf{test value} of $Q$, and call $Q$ a \textbf{preferable sample set} according to $\gamma_0$ if 
\begin{align}
\label{spL:est3}
T(Q)  \le \frac{32\sqrt{2}}{\pi \sqrt{\pi}} m^{1/4} {\gamma_0}^{3/4}  +   \frac{4\sqrt{2}}{\pi}  {{m}^{1/4}{\gamma_0}^{-1/4}}.
\end{align}
\end{definition}
Loosely speaking, preferable set conceptualizes our preference to sample sets with which $ \sum_{i=1}^m Z_i  $ is small. Observe that $Z_i$ is bigger when $\by_i$ approaches the endpoints ($\pm 1$), these sets slightly push sample points towards the center of the domain. Preferable sets according to large $\gamma_0$ (and small test value) should be less common. This intuition will be verified in Lemma \ref{lemma:good_set}, where we prove that $Q$ is a preferable sample set according to $\gamma_0$ with probability exceeding $1-\gamma_0$.  

The concept of preferable sets play a pivotal role in our analysis and is also useful in practice. We will show in both theory and computation that if we limit the sparse reconstruction to random Legendre matrices associated with preferable sample sets, better accuracy can be achieved. Enforcing this restriction is simple, as the computation cost of checking whether a sample set is preferable is negligible. Indeed, we can always generate multiple sets of $m$ random samples, and pick a preferable one to commence the reconstruction process. We could follow Definition \ref{def:PSS} strictly and choose a set that fulfills criterion \eqref{spL:est3}, but empirical approach is also equally effective: first ranking all the sample sets according to their test values, and then selecting the preferable set among those in lower percentile of test value distribution, say, among $(100(1-\gamma_0))\%$ of sample sets with smallest $ \sum_{i=1}^m Z_i$. 

%

\subsection{Main results}
\label{subsec:main_1d}
Our first main theorem is stated as follows. 
\begin{theorem}
\label{spL:WREP_theorem}
Let $\delta, \gamma$ and $\gamma_0$ be fixed parameters with 
$0<\delta < 1$, $0 < \gamma + \gamma_0 < 1$ and $\{L_{j}\}_{{j}\in \cJ}$ be a univariate Legendre orthonormal 
system of finite size $N = \#(\cJ)$. Denote $s_{\alpha,\delta} =  \frac{(1+\alpha)^2 s} {\delta^{3/2}} $ and assume that 
\begin{gather}
\label{spL:complexity}
\begin{aligned}
m \ge C (1+\alpha)^4{ s}^2  &\max\Big\{  \delta^{-12} \,   \gamma_0^{-1} \,{\log^2(s_{\alpha,\delta} )}  \log^2(2N),  
\\
& \delta^{-\frac{15}{2}} \,  \,{ \log^{\frac32}( s_{\alpha,\delta} )}  \log(2N),\ \ \delta^{-4 } \, { \log(s_{\alpha,\delta}) }\log\left( \frac{ { \log( s_{\alpha,\delta} )}}{\gamma \delta } \right) \Big\}.
 \end{aligned}
 \end{gather}
Let $\by_1,\by_2,\ldots,\by_m$ be drawn independently from the uniform distribution on $[-1,1]$. Then with probability exceeding $1-  (\gamma + \gamma_0)$, 
the normalized sampling matrix $\bA \in \mathbb{R}^{m\times {N}}$ (defined as in \eqref{defA}) satisfies
\begin{align}
  \|\bm{Az}\|_2^2 > (1 -  \delta) \|\bz\|^2_2,\ \ \forall \bm{z}\in {C}(s;\alpha). \label{note:problem1}
\end{align}
\end{theorem}

The above result shows that for Legendre matrices to satisfy restricted eigenvalue property, the number of samples needs to scale quadratically with the sparsity and only logarithmically with the size of the coefficient vectors, therefore, verifying successful uniform recovery from underdetermined Legendre systems and uniform sampling. The dependence of sample complexity on $\delta$ is severe ($\delta^{-12}$); however, as discussed in Section \ref{sec:WREP}, $\delta$ is allowed to approach $1$ with our use of restricted eigenvalue property, and the growth of error bound when $\delta$ tends to $1$ is mild. Recall $\delta_0 = 1/(1+\alpha)$, the dependence of the number of samples on $\delta_0$ is $\delta_0^{-4}$.   

A large part of our analysis will be conducted under the condition that the sample sets are preferable, giving rise to an additional parameter $\gamma_0$. In particular, the restricted eigenvalue property is guaranteed under a reduced probability $1-  (\gamma + \gamma_0)$. It is possible to restore the rate $1-\gamma$ (which is normally seen in the literature), if we limit the sparse reconstruction to random matrices associated with preferable sets (which is simple to do).  

\begin{proposition}
\label{spL:WREP_theorem_pref_set}
Let $\delta, \gamma$ and $\gamma_0$ be fixed parameters in $(0,1)$ and $\{L_{j}\}_{{j}\in \cJ}$ be a univariate Legendre orthonormal 
system of finite size $N = \#(\cJ)$.  Denote $s_{\alpha,\delta} =  \dfrac{(1+\alpha)^2 s} {\delta^{3/2}} $ and assume that
\begin{gather}
\label{spL:complexity_prefset}
\begin{aligned}
m \ge C (1+\alpha)^4{ s}^2  &\max\Big\{  \delta^{-12} \gamma_0^{-1} \,  \,{\log^2(s_{\alpha,\delta} )}  \log^2(2N),  
\\
& \delta^{-\frac{15}{2}} \,  \,{ \log^{\frac32}( s_{\alpha,\delta} )}  \log(2N),\ \ \delta^{-4 } \, { \log(s_{\alpha,\delta}) }\log\left( \frac{ { \log( s_{\alpha,\delta} )}}{(1-\gamma_0) \gamma \delta } \right) \Big\},
 \end{aligned}
 \end{gather}
 Let $Q = \{\by_1,\by_2,\ldots,\by_m\}$ be a set of samples drawn independently from the uniform distribution on $[-1,1]$. If $Q$ is a preferable set according to $\gamma_0$, then with probability exceeding $1-\gamma$, 
the normalized sampling matrix $\bA \in \mathbb{R}^{m\times {N}}$ (defined as in \eqref{defA}) satisfies
\begin{align}
  \|\bm{Az}\|_2^2 > (1 -  \delta) \|\bz\|^2_2,\ \ \forall \bm{z}\in {C}(s;\alpha). \label{note:problem1_prefset}
\end{align}
\end{proposition}

In Proposition \ref{spL:WREP_theorem_pref_set}, $\gamma_0$ can be considered a tuning parameter. Overall, the sample complexity is inversely proportional to $\gamma_0$; therefore, preferable sets according to high $\gamma_0$ allow reconstructions with lower number of samples and should be truly preferable. This fact is illustrated with numerical tests in Section \ref{sec:num_examples}. These sets are however increasingly uncommon as $\gamma_0$ tends to $1$. Also, $\gamma_0$ very close to $1$ may affect the third term of \eqref{spL:complexity_prefset} unfavorably. 

We will prove Theorem \ref{spL:WREP_theorem} and Proposition \ref{spL:WREP_theorem_pref_set} in Sections \ref{subsec:mainproof_1d} and \ref{sec:1d_prefset} correspondingly. First, some preparatory lemmas are needed. 

\subsection{Supporting lemmas}
\label{subsec:notation_1d}
We prove the following estimates involving the envelope bound $\Omega$. 
\begin{lemma}
\label{lemma:prob_tail}
For any $\mu >0$, there holds
$
\varrho\left( \by\in \cU: |\Omega(\by)| \ge \mu \right) \le \dfrac{16}{\pi^2 \mu^4}.
$
\end{lemma}
\begin{proof}
We have $ |\Omega(\by)| \ge \mu $ if and only if
$
\by^2 \ge 1- \dfrac{16}{\pi^2 \mu^4}.
$
There follows 
$$
\varrho\left( \by\in \cU: |\Omega(\by)| \ge \mu \right) \le 1 - \sqrt{1 - \frac{16}{\pi^2 \mu^4}} \le \frac{16}{\pi^2 \mu^4},
$$
as desired.
\end{proof}
\begin{lemma}
\label{lemma:end_set}
For any $0< \mu \le 1$ and $\overline{I}\subset \cU$ measurable with $\varrho(\overline{I}) \le \mu$, there holds 
$$
\int_{\overline{I}}  |\Omega(\by)|^2 d\varrho \le 2 \sqrt{\mu}.
$$
\end{lemma}
\begin{proof}
We have 
$$
\int_{\overline{I}}  |\Omega(\by)|^2 d\varrho \le \int_{1-\mu}^1 |\Omega(\by)|^2 d\by = \frac{4}{\pi}\(\frac{\pi}{2} - \arcsin(1-\mu)\). 
$$
Denoting $\kappa = ({\pi}/{2} - \arcsin(1-\mu))$, then 
$$
1 - \mu = \sin(\pi/2 - \kappa) = \cos(\kappa) \le 1 - 4 \kappa^2/\pi^2,
$$
which implies $\kappa \le \dfrac{\pi}{2}\sqrt{\mu}$. 
\end{proof}
\begin{lemma}
\label{lemma:imp_integral1}
For all $\beta > 0 $ and $0\le \tau < 1$, there holds 
\begin{align}
 \label{imp_integral1}
 & \int_{-1}^{1} \frac{\exp\( - {\beta}{\sqrt{1-\by^2}}\)}{({1-\by^2})^{\tau}} d\varrho \le 2\Gamma(2-2\tau) {\beta}^{2\tau -2 }.
\end{align}
Here, $\Gamma$ denotes the gamma function. 
\end{lemma}
\begin{proof}
Let us define $g(\by ) := \dfrac{\exp\( - {\beta}{\sqrt{1-\by^2}}\)}{({1-\by^2})^{\tau}} $. We have 
\begin{align*}
& \int\limits_{-1}^{1} g(\by ) d\varrho = \int\limits_{0}^{\infty} \varrho\left(\by \in \cU: g(\by ) \ge t\right) dt 
=  \int\limits_0^{\exp(-\beta)} dt + \int\limits_{\exp(-\beta)}^{\infty} \varrho\left(\by \in \cU:  g(\by ) \ge t\right) dt,  
\end{align*}
since $g(\by) \ge \exp(-\beta),\, \forall \by\in \cU $. By substituting $t = {\exp\( - {\beta}{\sqrt{u}}\)}/{u^{\tau}}$, it gives
  \begin{align}
  \int_{-1}^{1} g(\by ) d\varrho &= \exp(-\beta) + \int_{1}^{0} \varrho\(\by \in \cU:  g(\by ) \ge \frac{\exp( - {\beta}{\sqrt{u}})}{u^{\tau}} \) \( \frac{\exp( - {\beta}{\sqrt{u}})}{u^{\tau}}\)' du \notag
 \\
 & = \exp(-\beta) + \int_{1}^{0} \varrho\(\by \in \cU:  |\by| \ge \sqrt{1-u} \) \( \frac{\exp( - {\beta}{\sqrt{u}})}{u^{\tau}}\)' du .  \label{spL:est2b}
  \end{align}
  With notice that $ \varrho\(\by \in \cU:  |\by| \ge \sqrt{1-u} \) = 1-\sqrt{1-u} \le u$, there follows 
  \begin{align*}
  \int_{-1}^{1} g(\by ) d\varrho \le \exp(-\beta) + \int_1^0 u \( \frac{\exp( - {\beta}{\sqrt{u}})}{u^{\tau}}\)' du =  \int_0^1  \frac{\exp( - {\beta}{\sqrt{u}})}{u^{\tau}} du, 
  \end{align*}
  where the second equality comes from integration by parts. Substituting $v = \beta\sqrt{u}$ yields 
  \begin{align*}
   \int_{-1}^{1} g(\by ) d\varrho \le  \int_0^{\beta}  \frac{\beta^{2\tau }e^ {- v}}{v^{2\tau}} . \frac{2v}{\beta^2} dv = 2 \beta^{2\tau -2} \int_0^{\beta} v^{-2\tau +1} e^{-v} dv \le 2 \Gamma(2-2\tau) \beta^{2\tau -2}. 
  \end{align*}
\end{proof}

In the next lemma, we verify the probability of a set of samples from uniform distribution on $[-1,1]$ to be preferable.   
\begin{lemma}
\label{lemma:good_set}
Let $m>0,\, 0< \gamma_0 < 1$, and $Q = \{\by_1,\ldots,\by_m\}$ be a set of samples where $\by_1,\ldots,\by_m$ are drawn independently from the uniform distribution on $[-1,1]$. Then, $Q$ is a preferable sample set according to $\gamma_0$ with probability exceeding $1-\gamma_0$.  
\end{lemma}
\begin{proof}
Let $Z_i$ be defined as in \eqref{define:z}. By Lemma \ref{lemma:imp_integral1}, 
\begin{align}
\label{spL:est2}
\E Z_i \le \, \frac{2^{11/2}}{\pi \sqrt{\pi}} \left(\frac{\gamma_0}{m}\right)^{3/4}, \ \ \ \text{Var}(Z_i) \le \E |Z_i|^2 \le  \frac{2^5}{\pi^2} \left(\frac{\gamma_0}{m}\right)^{1/2}. 
\end{align}
Referring $\E Z_i$ and $\text{Var}(Z_i)$ as $\E Z$ and $\text{Var}(Z)$, we apply Chebyshev inequality to obtain
$$
\P\( | \sum_{i=1}^m Z_i - m\E Z | \ge \lambda  \) \le \frac{m}{\lambda^2} \text{Var}(Z). 
$$
Setting $\lambda =  \sqrt{\frac{m}{\gamma_0}} (\E |Z|^2)^{1/2}$, then with probability exceeding $1-\gamma_0$, 
\begin{align}
\label{spL:est5}
 \sum_{i=1}^m Z_i  \le m \E Z  +    \sqrt{\frac{m}{\gamma_0}} (\E |Z|^2)^{1/2}.
\end{align}
Substituting \eqref{spL:est2} to \eqref{spL:est5} yields the conclusion.  
\end{proof}

We also need some tail estimates, the proofs of which can be found in \cite{ChkifaDexterTranWebster15}.

\begin{lemma}
\label{lemma:hoeffding}
Let $X_1,\ldots,X_M$ be $M$ independent identically distributed real-valued 
random variables satisfying $|X_k| \le a$ and $\E[X_k] = X$ for all $k$. We denote 
$\overline{X} = \frac1M \sum_{k=1}^M X_k$. For every $\mu> 0$, 
\begin{align}
\label{est:hoeffding}
\P\(\left| \overline{X} - X \right| \ge \mu \) \le 2 \exp\({-\frac{M \mu^2}{ 2a^2}}\).
\end{align}   
\end{lemma}

{\begin{lemma}
\label{lemma:tailbound}
Let $X_1,\ldots,X_M$ be $M$ independent identically distributed Bernoulli 
random variables with $\E[X_k] = X$ for all $k$. Denote 
$\overline{X} = \frac1M\sum_{k=1}^M X_k$. Then, for every $0< \mu_1 < 1$, 
$\mu_2> 0 $ and $M \ge \frac{16e}{\mu_1\mu_2}$, there holds
\begin{align}
\label{est:tailbound}
\P\( |\overline{X} - X|  \ge \mu_1 X + \mu_2\) \le \exp\(- \frac{M\mu_1\mu_2}{16e}\).
\end{align}   
\end{lemma}

}

Let $
\cE_{s,\alpha} :=  \partial B_2 \cap C(s;\alpha), 
$
we conclude this subsection by proving the following important result involving covering number for $\cE_{s,\alpha}$. This step resembles but considerably extends the previous covering number estimates for $\ell_1$ balls, e.g., \cite{RudelsonVershynin08,Rauhut10,Bou14,ChkifaDexterTranWebster15}, in the sense that a much more complicated pseudo-metric is considered herein.

\begin{lemma}
\label{note:lemma4}
For $0 < \varsigma < 1$, $0 < \gamma_0 < 1$, $\mu >0$, and $m>0$, there 
exists a set $D \subset \R^N$ such that the following hold:
\begin{enumerate}
\item[(i)] If $Q = \{\by_1,\ldots,\by_m\}$ is a set of $m$ i.i.d. samples drawn from $(\cU,\varrho)$, preferable according to $\gamma_0$, then for all $\bz\in \cE_{s,\alpha}$, there exists $\bz' \in D$ depending on $\bz$ and $Q$ satisfying 
\begin{align}
|\psi(\by, \bm{z} - \bm{z}')| &\le \mu \quad\mbox{for all }\by \mbox{ in a subset }\cU^\star \mbox{ of } \cU\mbox{ with }\varrho(\cU^\star) > 1-\varsigma,  \label{metric1}
\\
|\psi(\by_i, \bm{z} - \bm{z}')| &\le \mu\quad \mbox{for all }\by_i \in Q. \label{metric2}
\end{align}

\item[(ii)] The cardinality of $D$ satisfies 
\begin{align}
\log(\#(D)) \le\,  C(1+\alpha)^2\frac{ s}{\mu^2} \log(2N) \max\left\{\frac{1}{\sqrt{ \varsigma }},\sqrt{\frac{m}{\gamma_0}}\right\} .
\end{align}
 
\end{enumerate}
\end{lemma} 

\begin{proof}
We will find $D$ using the empirical method of Maurey. First, we observe 
that $\cE_{s,\alpha} \subset (1+ \alpha){\sqrt{s}}B_{1}$ (see Lemma \ref{lemma:subset}), hence if we denote {$\mathcal{P} = \{\pm \bm{e}_j (1+\alpha) \sqrt{s}\}_{1\leq j\leq N}$}, 
where $(\bm{e}_j)$ are canonical unit vectors in $\R^N$, we have 
$\cE_{s,\alpha}  \subset \text{conv}(\mathcal{P})  $. 

Fix a ${\bm z}\in \cE_{s,\alpha}$, $\bz$ can be represented as ${\bm z} = \sum_{r=1}^{2N} \lambda_r \bm{v}_r$, for some $\lambda_r \ge 0$ such that $\sum_{r=1}^{2N} \lambda_r = 1$ and ${\bm v}_r$ listing $2N$ elements of $\mathcal{P}$.  
There exists a probability measure $\lambda$ on $\mathcal{P}$ that takes the 
values ${\bm v}_r \in \mathcal{P}$ with probability $\lambda_r$. Let $\bz_1,\ldots,\bz_M$ be i.i.d random variables with law $\lambda$. Note 
that $\E\bz_k = \bz$, for all $k =1,\dots, M$. For each $\by \in \mathcal{U}$, $\psi(\by,\bm{z}_k)$ 
is also a real-valued random variable on probability space 
$(\psi(\by,\mathcal{P}),\lambda)$ with 
$$
|\psi(\by,\bm{z}_k)| \le \Omega(\by)(1+\alpha)\sqrt{s},
\quad\quad\mbox{and}\quad\quad
\E\psi(\by,\bm{z}_k) = \psi(\by,\bm{z}). 
$$
Denote $\overline\bz =  \frac1M \sum_{k=1}^M \bz_k$. Let $D$ be 
the set of all possible outcomes of $\overline\bz $ and $\overline\lambda$ be the probability measure on $D$ according to $\overline\bz$. Note that $\#(D) \le (2N)^M$, we derive a sufficient condition on $M$ so that $D$ satisfies the assertion (i). 
 
First, let us define a characteristic function $\chi$ on $(\cU\times D,\varrho\otimes {\overline\lambda})$, given by  
\begin{align*}
\chi(\by,\overline\bz) = 
\begin{cases}
1,\qquad \mbox{if }\left|  \psi(\by,\overline{\bm{z}}-\bm{z}) \right| \ge \mu, \\ 
0,\qquad \mbox{if }\left|  \psi(\by,\overline{\bm{z}}-\bm{z}) \right| < \mu. 
\end{cases}
\end{align*}
Applying Lemma 
\ref{lemma:hoeffding} yields for all $\by \in \cU$,
$$
\int_D \chi(\by,\overline\bz)d\overline{\lambda} =  \P_{\overline\bz}\(\left|  \psi(\by,\overline{\bm{z}}-\bm{z}) \right| \ge \mu \) 
\le 2 \exp\({-\frac {M \mu^2}{ 2(1+\alpha)^2 \Omega^2(\by)  s}}\).
$$
From Lemma \ref{lemma:imp_integral1}, we have 
\begin{align}
\label{est7b}
 \int_D \left( \int_\cU \chi(\by,\overline\bz) d\varrho \right) d\overline{\lambda} =  \int_\cU \left( \int_D \chi(\by,\overline\bz)d\overline{\lambda} \right) d\varrho
\le   \frac{  C (1+\alpha)^4 s^2}{M^2\mu^4}. 
\end{align}
which yields that with probability exceeding 2/3, $\overline\bz\in D$ satisfying
\begin{align}
\label{est8}
 \P_{\by}\(\left| \psi(\by,\overline\bz - \bm{z}) \right| \ge \mu \) = 
 \int_\cU \chi(\by,\overline\bz) d\varrho   \le  \frac{3 C (1+\alpha)^4  s^2}{M^2\mu^4} \le \varsigma,
\end{align}
provided that $M \ge C (1+\alpha)^2 \dfrac{ s}{\mu^2 \sqrt{\varsigma}}$.

To deal with the set $Q = \{\by_1,\dots,\by_m\}$, we develop a different argument based on extending the approach in \cite[Lemma 12.37]{FouRau13}. First, by symmetrization, 
$$
\E_{\overline\bz}\(\max_{\by_i\in Q} |\psi(\by_i,\bz-\overline\bz )| \) \le \frac{2}{M} \E_{\overline\bz} \E_\epsilon \max_{\by_i\in Q} \left| \sum_{k=1}^M \epsilon_k  \psi(\by_i,\bz_k)  \right|, 
$$
where $\bm{\epsilon}$ is a Rademacher sequence independent of $(\bz_1,\ldots,\bz_M)$.
Since $|\psi(\by_i,\bz_k)|\le \Omega(\by_i)(1+\alpha) \sqrt{s}$ for all $\by_i\in Q,\, k\in[M]$, we have $\|(\psi(\by_i,\bz_k))_{k=1}^M\|_2$ $ \le \Omega(\by_i)(1+\alpha) \sqrt{Ms}$. Standard calculations yield
$$
\P_\epsilon\left(\left| \sum_{k=1}^M \epsilon_k  \psi(\by_i,\bz_k)  \right| \ge t \right) \le 2\exp\(-\frac{t^2}{ 2(1+\alpha)^2  \Omega^2(\by_i){Ms}}\),\quad \forall t>0. 
$$
By the union bound, 
$$
\P_\epsilon\left(\max_{\by_i\in Q}\left| \sum_{k=1}^M \epsilon_k  \psi(\by_i,\bz_k)  \right| \ge t \right) \le \sum_{i=1}^m 2 \exp\(-\frac{t^2}{ 2 (1+\alpha)^2  \Omega^2(\by_i){Ms}}\),\quad \forall t>0. 
$$
We have for every $\kappa > 0 $
\begin{align}
& \E_\epsilon \max_{\by_i\in Q} \left| \sum_{k=1}^M \epsilon_k  \psi(\by_i,\bz_k)  \right| = \int_{0}^{\infty} \P_\epsilon\left(\max_{\by_i\in Q}\left| \sum_{k=1}^M \epsilon_k  \psi(\by_i,\bz_k)  \right| \ge t \right) dt \notag
\\
\le &  \int_0^{\kappa} 1 \ dt + 2 \sum_{i=1}^m \int_{\kappa}^{\infty}  \exp\(-\frac{t^2}{ 2 (1+\alpha)^2   \Omega^2(\by_i){Ms}}\) dt \notag
\\
 \le & \,  \kappa \, + \,  \sqrt{2 \pi Ms} (1 + \alpha) \sum_{i=1}^m { \Omega(\by_i)} \exp\(-\frac{\kappa^2}{2 (1+\alpha)^2  \Omega^2(\by_i){Ms}}\). \label{spL:est1}
\end{align}
Choosing $\kappa =  (1+\alpha) \sqrt{Ms} \sqrt[4]{{m}/{\gamma_0}} $ leads to
\begin{gather}
 \label{spL:est1b}
\begin{aligned}
& \E_\epsilon \max_{\by_i\in Q} \left| \sum_{k=1}^M \epsilon_k  \psi(\by_i,\bz_k)  \right|  \le  \,  (1+\alpha) \sqrt{Ms} \sqrt[4]{\dfrac{m}{\gamma_0}} \, 
\\
& \qquad\qquad\qquad\qquad  + \,   \sqrt{2 \pi Ms} (1+ \alpha) \sum_{i=1}^m { \Omega(\by_i)} \exp\(-\frac{1}{2 \Omega^2(\by_i)}  \sqrt{\dfrac{m}{\gamma_0}} \).
\end{aligned}
\end{gather}
Since $Q$ is a preferable sample set, we obtain from \eqref{spL:est3}
\begin{align*}
\E_\epsilon \max_{\by_i\in Q} \left| \sum_{k=1}^M \epsilon_k  \psi(\by_i,\bz_k)  \right| & \le C (1+ \alpha) \sqrt{Ms} ({m}^{\frac14}{\gamma_0}^{-\frac{1}{4}} + m^{\frac14 }\gamma_0^{\frac34}) \le  C (1+ \alpha) \sqrt{Ms}  \sqrt[4]{\dfrac{m}{\gamma_0}}.
\end{align*}
There follows $\E_{\overline\bz}\(\max_{\by_i\in Q} |\psi(\by_i,\bz - \overline\bz)| \) \le  C (1+ \alpha) \sqrt{\dfrac{s}{M}}  \sqrt[4]{\dfrac{m}{\gamma_0}}$, yielding with probability exceeding $2/3$, $\overline \bz \in D$ satisfies
\begin{align}
\label{spL:est4}
\max_{\by_i\in Q} |\psi(\by_i,\bz - \overline\bz)| \le   3C (1+ \alpha) {\sqrt{\frac{s}{M}}} {\sqrt[4]{\frac{m}{\gamma_0}}} \le \mu,
\end{align}
provided that $M \ge  \dfrac{C (1+\alpha)^2 s\sqrt{m}}{\mu^2\sqrt{\gamma_0}}$. Thus, for $M \ge \dfrac{C(1+\alpha)^2 s}{\mu^2} \max\left\{\dfrac{1}{ \sqrt{\varsigma }}, \dfrac{\sqrt{m}}{\sqrt{\gamma_0}}\right\}$, there exists $\overline{\bz}\in D$ that fulfills both \eqref{est8} and \eqref{spL:est4}. As $\log(\#(D))\le M\log(2N)$, the proof is concluded. 
\end{proof}

\subsection{Proof of Theorem \ref{spL:WREP_theorem}}
\label{subsec:mainproof_1d}

It is enough to prove \eqref{note:problem1} for all $\bz\in\cE_{s,\alpha} $ (i.e., $\partial B_2 \cap C(s;\alpha)$). We conduct the analysis under assumption $\delta < 1/13$ for convenience; this parameter will be rescaled at the end. For $m\in \mathbb{N}$, let $Q = \{\by_1,\dots,\by_m\}$ be a set of samples where $\by_1,\ldots,\by_m$ are drawn independently from the probability space $(\cU,\varrho)$. From Lemma \ref{lemma:good_set}, $Q$ is preferable according to $\gamma_0$ with probability exceeding $1-\gamma_0$. We define the set of integers 
\begin{align} 
\cL =  \Z\cap \(\frac{\log(\delta)}{\log(1+\delta)}+1, 
\dfrac{\log(\Xi)}{\log(1+\delta)}+1\),\ \mbox{ where }\Xi =  \frac{4(1+\alpha)^2 s }{\sqrt{3 \pi \delta}},
\label{defineIndices}
\end{align} 
and denote by $ \underline l, \overline l$ the minimum and maximum of 
$\cL$ respectively. $\underline l$ 
and $\overline l$ then satisfy 
\begin{align} 
(1+\delta)^{\underline l -2} \leq \delta 
\quad\mbox{and}\quad
(1+\delta)^{\overline l -1} \leq \Xi \leq (1+\delta)^{\overline l }. 
\label{maxmin}
\end{align} 
  
\textbf{Step 1:} under the condition that $Q$ is a preferable sample set according to $\gamma_0$ (pref. a.t. $\gamma_0$), for $0<\varsigma <1$ (whose exact value will be set later), we seek to construct $\tilde{\psi}$ approximating $\psi$ such that
\begin{enumerate}
\item For all $\bm{z} \in \cE_{s,\alpha} $, the following holds with probability exceeding $1 - \varsigma$ in $(\cU,\varrho)$ and \textit{for all} $\by \in Q$
\begin{gather}
\begin{aligned} 
&\!\! \left(1-{3\delta}/{2}\right)\tilde{\psi}(\by,\bm{z})  < |{\psi}(\by,\bm{z})| <  
\left(1+ {3\delta}/{2}\right)\tilde{\psi}(\by,\bm{z}) , &\mbox{ if }\tilde{\psi}(\by,\bm{z}) > 0,
\\ 
&\!\! \ |{\psi}(\by,\bm{z})| < 6\delta/5 \ \ \mbox{ or }\ \ |{\psi}(\by,\bm{z})| > \Xi, &\mbox{ if }\tilde{\psi}(\by,\bm{z}) = 0.
\end{aligned}
\label{note:eq1b}
\end{gather}
\item For all ${\bm z}\in \cE_{s,\alpha}$, there exists a pairwise disjoint family of subsets $(I_l)_{l\in \mathcal{L}} = (I^{({\bm z},Q)}_l)_{l\in \mathcal{L}}$ of $\cU$ depending on $\bm{z}$ and $Q$ such that 
\begin{align}
\label{note:eq2}
\tilde{\psi}(\cdot,{\bm z}) = \sum_{l\in\mathcal{L}} (1+\delta)^{l} \chi_{I_l}. 
\end{align}
\item For every $l \in \mathcal{L}$, $(I_l^{({\bm z}, Q)})_{\substack {{\bm z}\in \cE_{s,\alpha} \\  Q\mbox{ {\scriptsize pref. a.t. $\gamma_0$} }} }$ belongs to a finite class $F_l$ of subsets of $\cU$ satisfying 
\begin{align}
\label{note:eq2b}
\log(\#F_l) \le C\frac{ (1+\alpha)^2 s}{\delta^3 (1 + \delta)^{2l-2} } \log(2N) \max\left\{{\sqrt{ \dfrac{ \log( \Xi /\delta)} {\varsigma \log(1+\delta)} }},\sqrt{\frac{m}{\gamma_0}}\right\}. 
\end{align}
\end{enumerate}

First, for $l\in \mathcal{L}$, let $D_l$ be a finite subset of $\mathbb{R}^N$ determined as in Lemma \ref{note:lemma4} with given $m,\, \gamma_0$, as well as $\mu =  {{\delta}(1+\delta)^{l-1}}/{2}$ and $0 < \varsigma'< 1$ (to be set accordingly to meet our needs). We have 
\begin{align}
\log(\# D_l ) & \le C\frac{ (1+\alpha)^2 s}{\delta^2 (1 + \delta)^{2l-2} } \log(2N) \max\left\{\frac{1}{\sqrt{ \varsigma' }},\sqrt{\frac{m}{\gamma_0}}\right\} .
\label{note:cardD}
\end{align}
For a fixed ${\bm z}\in \cE_{s,\alpha}$, there exist $\bm{z}_l \in D_l$ and a measurable set $\cU_l \subset \cU$ with $\varrho(\cU_l) \ge 1-\varsigma'$  such that 
\begin{align*}
|\psi(\by, \bm{z} - \bm{z}_l)| &\le  {{\delta}(1+\delta)^{l-1}}/{2},\quad \forall \by \in \cU_l, 
\\
|\psi(\by_i, \bm{z} - \bm{z}_l)| &\le  {{\delta}(1+\delta)^{l-1}}/{2},\quad \forall \by_i \in Q. 
\end{align*} 
Without loss of generality, we can assume $Q \subset \cU_l$. We construct a pairwise disjoint family of subsets $(I_l)_{l\in \cL}$ and mapping $\tilde{\psi}(\cdot,\bm{z}):\mathcal{U} \to \mathbb{R}$ which depend on $\bz$ and $Q$, inductively for the integers 
$\overline l > \dots > \underline l$ according to 
\begin{gather}
\begin{aligned}
& I'_l  = \{\by \in \cU: (1+\delta)^{l-1} < |\psi(\by,{\bm z}_l)| < (1+\delta)^{l+1}\}, 
\\
& I_l  = I'_l \,  {\setminus} \, \bigcup_{r>l} I'_{r},
\\
& \tilde{\psi}(\cdot,{\bm z}) = \sum_{l\in \mathcal{L}} (1+\delta)^{l} \chi_{I_l}. 
\end{aligned}
\label{constructionPsiZ}
\end{gather}

We proceed to prove that $\tilde{\psi}$ satisfies \eqref{note:eq1b}--\eqref{note:eq2b}, following closely the argument in \cite[Theorem 2.2]{ChkifaDexterTranWebster15}.  First, consider $\by\in \bigcap_{l\in \cL} \cU_l$. If $\by\in I_l$ for some $l\in \cL$, then 
\begin{align*}
\tilde{\psi}(\by,{\bm z}) = (1+\delta)^{l} > 0 \mbox{ and }(1+\delta)^{l-1} < |\psi(\by,\bm{z}_l)| < (1+\delta)^{l+1}.
\end{align*}
Since $| |\psi(\by,{\bm z})| - |\psi(\by,{\bm z}_l)| | \le |\psi(\by,\bm{z}) - \psi(\by,\bm{z}_l)| \le {\delta}(1+\delta)^{l-1}/2$, we have 
\begin{align*}
|\psi(\by,{\bm z})| & < (1+\delta)^{l+1} + \frac{\delta}{2}(1+\delta)^{l-1}   < \left(1+\frac{3}{2}\delta\right)\tilde{\psi}(\by,{\bm z}), 
\\
\text{ and }|\psi(\by,{\bm z})| & > (1+\delta)^{l-1} - \frac{\delta}{2}(1+\delta)^{l-1} > \left(1-\frac{3\delta}{2}\right)\tilde{\psi}(\by,{\bm z}). 
\end{align*}
If $ \by \notin \bigcup_{l\in \cL} I_l$, then $\tilde{\psi}(\by,\bm{z}) = 0$ and for every $l\in \mathcal{L}$, 
\begin{align*}
  |\psi(\by,\bm{z}_l)| \notin ((1+\delta)^{l-1},(1+\delta)^{l+1}). 
\end{align*}
With notice that $| |\psi(\by,{\bm z})| - |\psi(\by,{\bm z}_l)| | < {\delta}(1+\delta)^{l-1} /2 $, there follows
\begin{align*}
|\psi(\by,{\bm z})| \notin \bigcup_{l\in \mathcal{L}}\left((1+\frac{\delta}{2})(1+\delta)^{l-1} , (1+\frac{3\delta}{2} + \delta^2)(1+\delta)^{l-1}  \right).
\end{align*}
Observe that $(1+\frac{3\delta}{2} + \delta^2)(1+\delta)^{l-1}  > (1+\frac{\delta}{2})(1+\delta)^{l}$, the previous intervals intersect for any two consecutive values of $l$. We infer
\begin{align*}
|\psi(\by,{\bm z})| \le (1+\frac{\delta}{2})(1+\delta)^{\underline{l}-1},\ \mbox{ or }\ |\psi(\by,{\bm z})| \ge (1+\frac{3\delta}{2} + \delta^2)(1+\delta)^{\overline{l}-1}. 
\end{align*}
It implies by \eqref{maxmin} and assumption $\delta< 1/13$ that $|\psi(\by,\bz)| \leq \delta(1+\delta/2)(1+\delta) < 6\delta/5$ (if the first inequality occurs) or $|\psi(\by,\bz)| > \Xi $ (if the second inequality occurs). We complete this case by emphasizing $Q \subset \bigcap_{l\in \cL} \cU_l $. 

Next, consider $\by\notin \bigcap_{l\in \cL} \cU_l$, \eqref{note:eq1b} is not guaranteed. However, this only holds with probability not exceeding
\begin{gather}
\begin{aligned}
 \varrho\biggl(\cU \setminus  \bigcap_{l\in \mathcal{L}} \cU_l \biggl) \le \sum_{l\in \cL}\varrho(\cU\setminus  \cU_l) & \le \varsigma' (\#\mathcal{L}) \le  \frac{ \log( \Xi / \delta)}{\log(1+\delta)} \varsigma' = \varsigma,
\end{aligned}
 \label{note:badset}
\end{gather}
for the last equality we set $\varsigma' =  \dfrac{\log(1+\delta)}{ \log( \Xi / \delta)} \varsigma$. 

To summarize, we partitioned $\cU$ into three sets 
\vspace{.2cm}

\qquad $I := \left(\bigcap_{l\in \cL} \cU_l \right) \bigcap \left(\bigcup_{l\in \cL} I_l\right),\ \ \widehat{I} := \left(\bigcap_{l\in \cL} \cU_l \right) \setminus \left(\bigcup_{l\in \cL} I_l\right),$ 
\vspace{.1cm}

\qquad $\cU' := \cU\setminus \left(\bigcap_{l\in \cL} \cU_l\right), $
\vspace{.2cm}
\\
and constructed $\tilde \psi(\cdot, \bz)$ depending on $\bz$ and $Q$ approximating $\psi(\cdot,\bz)$ satisfying 
\begin{align} 
&\left(1-{3\delta}/{2}\right)\tilde{\psi}(\by,\bm{z})  < |{\psi}(\by,\bm{z})| <  \left(1+ {3\delta}/{2}\right)\tilde{\psi}(\by,\bm{z}) ,\ & \forall \by \in I, \label{note:partition1}
\\
& (0\le |{\psi}(\by,\bm{z})| < 6\delta/5\ \vee \ |\psi(\by,\bz)| > \Xi) \ \ \mbox{ and }\ \ \tilde{\psi}(\by,\bm{z}) = 0 ,\ & \forall \by \in \widehat{I}, \label{note:partition2}
\\
& \varrho(\cU' )\le \varsigma,\qquad\qquad Q\subset I\cap \widehat{I},\mbox{ i.e., }Q\cap \cU' = \emptyset. &  \label{note:partition3}
\end{align}
We further divide $\widehat{I}$ into two subsets $\underline{I}$ and $\overline{I}$ according to 
\begin{align}
\underline{I}:= \{\by \in \widehat{I}:  |{\psi}(\by,\bm{z})| < 6\delta/5\}, \qquad \overline{I} := \{\by \in \widehat{I}:  |{\psi}(\by,\bm{z})| > \Xi \},  \label{note:partition4}
\end{align}
as they will need different treatments. Note that for all $\by \in \overline{I}$, $|\Omega(\by)| \ge \dfrac{|\psi(\by,\bz)|}{(1+\alpha)\sqrt{s}} > \dfrac{\Xi}{(1+\alpha)\sqrt{s}}$, thus by Lemma \ref{lemma:prob_tail}, $\varrho(\overline{I}) \le \dfrac{16(1+\alpha)^4 s^2}{\pi^2 \Xi^4} = \dfrac{9 \delta^2}{16(1+\alpha)^4 s^2}$.

It remains to verify \eqref{note:eq2b}. For any $l\in \mathcal{L}$, $\#\{I'^{({\bm z},Q)}_l\, |\, {\bm z}\in \cE_{s,\alpha},\, Q \mbox{ is preferable}\} \le \#D_l$ and 
$\# F_l \le \prod_{r \ge l} \# D_{r}$. From \eqref{note:cardD}, it gives  
\begin{align*}
\log(\# F_l)  &\le \sum_{r\ge l} \log(\# D_{r}) \le  C \frac{(1+\alpha)^2 s}{\delta^3 (1 + \delta)^{2l-2} } \log(2N) \max\left\{{\sqrt{ \dfrac{ \log( \Xi / \delta)} {\varsigma \log(1+\delta)}  }},\sqrt{\frac{m}{\gamma_0}}\right\} .
\end{align*}

\textbf{Step 2:} We derive estimates of $\|\bA \bz\|_2$ and $\| \bz\|_2$ in terms of $\tilde{\psi}(\cdot,\bz)$. The following bounds will be useful for this task. First, from \eqref{note:partition1}, 
\begin{align}
  (1- 3\delta)|\tilde{\psi}(\by,{\bm z})|^2 <  |{\psi}(\by,{\bm z})|^2 <  (1+ 4\delta)|\tilde{\psi}(\by,{\bm z})|^2,\ \forall \by\in I. \label{spL:est6}
\end{align}
\eqref{spL:est6} gives $ |{\tilde \psi}(\by,{\bm z})|^2 <  (1+ 4\delta)|{\psi}(\by,{\bm z})|^2$, which implies 
\begin{align}
\label{spL:est12}
\int_{I} |\tilde \psi(\by,\bz)|^2 d\varrho \le (1+4\delta) \int_{I} | \psi(\by,\bz)|^2 d\varrho \le  (1+4\delta) \int_{\cU} | \psi(\by,\bz)|^2 d\varrho  = 1+ 4\delta.
\end{align}

A lower bound of $\|\bA \bz\|_2$ in terms of $\tilde{\psi}(\cdot,\bz)$ is straightforward. Since $Q \cap \cU' = \emptyset$ and $\tilde{\psi}(\by,\bz) = 0,\, \forall \by\in \widehat{I}$, we have 
\begin{align}
\|\bm{Az}\|^2_2 & = \frac{1}{m}\sum_{i=1}^m |\psi(\by_i,{\bm z})|^2 \ge {(1-{3\delta})} \sum_{i=1}^m \frac {| \tilde{\psi}(\by_i,{\bm z})|^2}{m} . \label{spL:est7}
\end{align}
It is worth mentioning that to obtain a reasonable upper estimate of $\|\bA \bz\|_2$ in similar manner is not easy, due to the lack of uniform bound for the considered orthonormal systems. However, unlike RIP, upper estimate of $\|\bA \bz\|_2$ is not needed for restricted eigenvalue property. 

For $\|\bz\|_2$, we decompose
\begin{align}
\label{spL:est8}
\|\bz\|_2^2 = \int_{I}\! |\psi(\by,\bz)|^2 d\varrho + \int_{\underline{I}}\! |\psi(\by,\bz)|^2 d\varrho + \int_{\overline{I}} \! |\psi(\by,\bz)|^2 d\varrho + \int_{\cU'} |\psi(\by,\bz)|^2 d\varrho,
\end{align}
and seek to produce upper bounds for each term in the RHS of \eqref{spL:est8}. To begin with, by \eqref{spL:est6},
\begin{align}
\label{spL:est9}
 \int_{I} |\psi(\by,\bz)|^2 d\varrho \le (1+4\delta) \int_{I} |\tilde{\psi}(\by,\bz)|^2 d\varrho.
\end{align}
By \eqref{note:partition4} and Lemma \ref{lemma:end_set}, note that $\varrho(\underline{I})\le 1$ and $\varrho(\overline{I})\le \dfrac{9\delta^2}{16(1+\alpha)^4 s^2}$, it gives 
\begin{gather}
 \label{spL:est10}
\begin{aligned}
&\int_{\underline{I}} |\psi(\by,\bz)|^2 d\varrho \le (6\delta/5)^2 \varrho(\underline{I}) \le \delta/6, 
\\
& \int_{\overline{I}} |\psi(\by,\bz)|^2 d\varrho \le (1+\alpha)^2 s  \int_{\overline{I}} |\Omega(\by)|^2 d\varrho  \le \frac{3 \delta}{2}. 
\end{aligned}
\end{gather}
Similarly, setting $\varsigma = \dfrac{\delta^2} {36 (1+ \alpha)^4 s^2}$, we have 
\begin{align}
\label{spL:est11}
\int_{\cU'} |\psi(\by,\bz)|^2 d\varrho  \le (1+\alpha)^2 s  \int_{\cU'} |\Omega(\by)|^2 d\varrho < \frac{\delta}{3}.
\end{align}
We combine \eqref{spL:est8}-\eqref{spL:est11} to get 
$$
\|\bz\|_2^2 <   2\delta + (1+4\delta) \int_{I} |\tilde{\psi}(\by,\bz)|^2 d\varrho, 
$$
which in reference to \eqref{spL:est12} and \eqref{spL:est7} implies 
\begin{gather}
\label{spL:est13}
\begin{aligned}
&\|\bz\|_2^2 - \|\bA \bz\|_2^2  < 2\delta +  (1+4\delta) \int_{I} |\tilde{\psi}(\by,\bz)|^2 d\varrho -  {(1-{3\delta})} \sum_{i=1}^m \frac {| \tilde{\psi}(\by_i,{\bm z})|^2}{m}
\\
\le & \ 2\delta + 7\delta  \int_{I} |\tilde \psi(\by,\bz)|^2 d\varrho  + (1 -  3\delta) \! \left( 
 \int_{\cU}\! |\tilde \psi(\by,\bz)|^2 d\varrho - \sum_{i=1}^m \frac{| \tilde{\psi}(\by_i,{\bm z})|^2}{m} \right)   
 \\
 \le &  \ \frac{34\delta}{3} +  (1 -  3\delta) \! \left( 
 \int_{\cU}\! |\tilde \psi(\by,\bz)|^2 d\varrho - \sum_{i=1}^m \frac{| \tilde{\psi}(\by_i,{\bm z})|^2}{m} \right).   
\end{aligned}
\end{gather}

\textbf{Step 3:} we derive a positive upper bound of $ \|\bz\|_2^2 - \|\bA\bz\|_2^2 $ via \eqref{spL:est13}, by employing the tail estimate in Lemma \ref{lemma:tailbound} and union bound. From the definition of $\tilde \psi$, it is easy to see that
\begin{align}
\label{note:comp4}
\int_{\cU} |\tilde \psi(\by,\bz)|^2 d\varrho
-
 \sum_{i=1}^m \frac{| \tilde{\psi}(\by_i,{\bm z})|^2}{m} 
=
\sum_{l\in\cL}  
(1+\delta)^{2l}\!
\left(
\varrho(I_l)
-
\frac{\#(Q\cap I_l)}m
\right).
\end{align}
Let $(\kappa_l)_{l\in\mathcal{L}}$ be a sequence of positive numbers. For any set $\Delta$ in the class $F_l$, for $Q$ being a set of samples $\by_1,\ldots,\by_m$ drawn independently from the probability space $(\cU,\varrho)$ (not necessarily preferable), applying Lemma \ref{lemma:tailbound} yields with probability of $Q$ exceeding $1-\exp\left(- \frac{m \kappa_l \delta}{16e}\right)$,
\begin{align}
  \varrho (\Delta) -  \frac{ \#( Q\cap \Delta)}{m}  \le  \,\delta \varrho (\Delta)  +  \kappa_l. \label{note:eq6}
\end{align}
By the union bound, with probability exceeding $1 - \sum_{l\in \mathcal{L}} \exp\left(- \frac{m \kappa_l \delta}{16e}\right) (\# F_l) $, the previous inequality holds uniformly for all sets $\Delta \in \cup_{l\in \cL} F_l$. Therefore, with probability exceeding $1 - \gamma_0 - \sum_{l\in \mathcal{L}} \exp\left(- \frac{m \kappa_l \delta}{16e}\right) (\# F_l) $, $Q$ is preferable according to $\gamma_0$ \textit{and} \eqref{note:eq6} holds with $Q$ uniformly for $\Delta \in \cup_{l\in \cL} F_l$. In this scenario, we can apply \eqref{note:eq6} with $\Delta = I^{(\bz,Q)}_l$ ($l\in \cL$) to the sum in \eqref{note:comp4} and combine with \eqref{spL:est13} to infer that for all $\bz\in \cE_{s,\alpha}$, 
\begin{gather}
\label{spL:est14}
\begin{aligned}
\|\bz\|_2^2 -  \|\bA\bz\|_2^2 
 &<  
\frac{34\delta}{3} + (1-3\delta)\sum_{l\in\cL}  
(1+\delta)^{2l}\!
\left(
\delta \varrho (I_l)  +  \kappa_l
\right)
\\
& = \frac{34\delta}{3} + \delta (1- 3\delta)\int_{\cU} |\tilde \psi(\by,\bz)|^2 d\varrho 
+ 
 (1-3\delta)\sum_{l\in\cL}  
(1+\delta)^{2l}
  \kappa_l
 .
\end{aligned}
\end{gather}
Note that $\tilde\psi(\by,\bz) = 0,\, \forall \by\in \widehat{I}$ and $\tilde\psi(\by,\bz) \le (1+\delta)^{\overline{l}} \le (1+\delta)\Xi,\, \forall \by\in \cU$, we estimate 
\begin{align*}
& \int_{\cU} |\tilde \psi(\by,\bz)|^2 d\varrho = \int_{I} |\tilde \psi(\by,\bz)|^2 d\varrho + \int_{\cU'} |\tilde \psi(\by,\bz)|^2 d\varrho 
\\
 \le\ & (1+4\delta) \int_{I} | \psi(\by,\bz)|^2 d\varrho + \varrho(\cU') (1 + \delta)^2 \Xi^2 < 1+ 5\delta
\end{align*}
and obtain 
\begin{align*}
\|\bz\|_2^2 -  \|\bA\bz\|_2^2 < 
\frac{25\delta}{2} +  
 (1-3\delta)\sum_{l\in\cL}  
(1+\delta)^{2l}
  \kappa_l.
 \end{align*}

Finally, in order to obtain Theorem \ref{spL:WREP_theorem}, we 
need to assign appropriate values for $\kappa_l$ and derive 
conditions on $m$ such that 
\begin{align*}
\sum_{l\in\cL}  
(1+\delta)^{2l}
\kappa_l  \le \delta/2,
\quad\quad\mbox{and}\quad\quad
\sum_{l\in \cL}\exp\(- \frac{m\kappa_l \delta }{16 e} + \log(\# F_l)\)\le \gamma.
\end{align*}
The two inequalities can be fulfilled if the numbers 
$\kappa_l$ and the integer $m$ are chosen as follows  
\begin{align*}
\kappa_l :=\frac{\delta/2}{(\#\cL) (1+\delta)^{2l}},\quad\quad
- \frac{m\kappa_l \delta }{16 e} + \log(\# F_l) 
\le \log\left(\frac{\gamma}{\#\cL}\right),\quad \quad  l \in \cL.
\end{align*}
This implies that 
\begin{align*}
m & \ge  32 e~(\#\cL) \frac{(1+\delta)^{2l}}{\delta^2} \left[ \log(\# F_l)  
+ \log\left(\frac{\#\cL}{\gamma}\right)\right],\ \quad\quad l\in \cL. 
\end{align*}
We have in view 
of \eqref{note:eq2b} and assumption $\delta < 1/13$ that

\begin{align*}
&32 e\,  (\#\cL) \frac{(1+\delta)^{2l}}{\delta^2} \log(\# F_l)  
\\
& \qquad \le C \frac{ \log( \Xi / \delta)}{\log(1+\delta)} \cdot \frac{(1+\delta)^{2l}}{\delta^2} \cdot \frac{(1+\alpha)^2 s}{\delta^3 (1 + \delta)^{2l-2} } \log(2N) \max\left\{{\sqrt{ \dfrac{ \log( \Xi / \delta)} {\varsigma \log(1+\delta)}  }},\sqrt{\frac{m}{\gamma_0}}\right\}
\\
&  \qquad \le  C  \delta^{-6} \,  (1+\alpha)^2{ s} \,{ \log( \Xi / \delta)}  \log(2N) \max\left\{  (1+\alpha)^2 s{\sqrt{ \dfrac{ \log( \Xi / \delta)} {\delta^3}  }}  ,\sqrt{\frac{m}{\gamma_0}}\right\} , 
\\
&32 e\,  (\#\cL) \frac{(1+\delta)^{2l}}{\delta^2}  \log\left(\frac{\#\cL}{\gamma}\right) \le   C \frac{ \log( \Xi / \delta)}{\log(1+\delta)} \cdot  \frac{(1+\delta)^{2} \Xi^2 }{\delta^2}   \cdot   \log\left( \frac{ \log( \Xi / \delta)}{\gamma \log(1+\delta)}\right) 
 \\
&\qquad  \le C  \delta^{-4 }{{(1+\alpha)^4 s^2 }} \, { \log( \Xi / \delta) }\log\left( \frac{ \log( \Xi / \delta)}{\gamma \delta } \right). 
\end{align*} 
Combining the two estimates shows that $m$ in \eqref{spL:complexity} guarantees 
$$
\|\bA \bz\|_2^2 > (1 - 13\delta)\|\bz\|_{2}^2, \ \ \forall \bm{z}\in \cE_{s,\alpha},
$$
with probability exceeding $1-(\gamma + \gamma_0)$. Rescaling $\delta$ concludes the proof. $\square$

\subsection{Proof of Proposition \ref{spL:WREP_theorem_pref_set}}
\label{sec:1d_prefset}
The proof follows closely that of Theorem \ref{spL:WREP_theorem}, with one key modification in Step 3. First, we note that Steps 1 and 2, in particular \eqref{spL:est13} and \eqref{note:comp4}, were established under the condition that $Q$ is preferable according to $\gamma_0$, which holds with probability exceeding $1-\gamma_0$. Independently, for $Q$ being a set of $m$ samples drawn uniformly from $[-1,1]$ (not necessarily preferable), with probability exceeding $1-\gamma$, \eqref{note:eq6} holds uniformly for all sets $\Delta \in \cup_{l\in \cL} F_l$, assuming $m$ satisfying $ \sum\limits_{l\in \mathcal{L}} \exp\left(- \dfrac{m \kappa_l \delta}{16e}\right) (\# F_l)\le \gamma$. Hence, given $Q$ preferable according to $\gamma_0$, \eqref{note:eq6} holds uniformly for all sets $\Delta \in \cup_{l\in \cL} F_l$ with probability exceeding $1 - \dfrac{\gamma}{1-\gamma_0}$. 
%
%
With \eqref{spL:est13}--\eqref{note:eq6} in hand, proceeding as in Section \ref{subsec:mainproof_1d}, we can derive that under the condition \eqref{spL:complexity}, with probability exceeding $1 - \dfrac{\gamma}{1-\gamma_0}$, 
\begin{align}
  \|\bm{Az}\|_2^2 > (1 -  \delta) \|\bz\|^2_2,\ \ \forall \bm{z}\in {C}(s;\alpha). \label{note:problem1}
\end{align}
By a change of variable, the proposition is concluded.  $\square$

\section{Restricted eigenvalue property for multivariate Legendre systems}
\label{sec:mult}We present an estimate for restricted eigenvalue property for multi-dimensional Legendre matrices. The overall strategy in this section is very similar to those in Section \ref{subsec:main_1d}, but complete analysis involves intensive multivariate calculations. We will focus on such calculations and skip repetitive details whenever possible. Let $\cU := [-1,1]^d$, $\varrho$ be the uniform probability measure on $\cU$, $\cJ$ be a subset of $\mathbb{N}^d$, and $\{\Psi_{j}\}_{{j}\in \cJ} \equiv \{L_{j}\}_{{j}\in \cJ}$, the system of multivariate Legendre polynomials orthonormal with respect to $\varrho$. We again denote
\begin{align*}
\psi(\by,\bz) &:= \sum_{j \in \cJ} {z}_{j} L_{j}(\by)  ,\quad\quad \by\in \mathcal{U},\;\; \bz \in \R^N.
\end{align*}
From the envelope bound of univariate Legendre polynomials, we have 
\begin{align}
\label{envelope}
|L_{j}(\by)| \le \frac{{2^{d}}}{\pi^{d/2}\prod_{k=1}^d{(1-y_k^2)^{1/4}}},  
\end{align}
for all $\by = (y_1,\ldots,y_d) \in (-1,1)^d,\ j \in \mathbb{N}^d$. Therefore, we define 
$$
\Omega (\by) := \frac{{2^{d}}}{\pi^{d/2}\prod_{k=1}^d{(1-y_k^2)^{1/4}}}.
$$

Let us introduce the notion of {preferable set} in multi-dimensional domain.

\begin{definition}
For $m > 0$, $0 < \gamma_0 < 1$, let $Q = \{\by_1,\ldots,\by_m\}$ be a set of samples where $\by_1,\ldots,\by_m$ are drawn independently from the uniform distribution on $[-1,1]^d$. Define $m$ i.i.d random variables 
\begin{align}
\label{define:z}
Z_i =  { \Omega(\by_i)} \exp\(-\dfrac{1}{2 \Omega^2(\by_i)}  \sqrt{\dfrac{m}{\gamma_0}} \),\quad \forall i \in \{1,\ldots, m\}, 
\end{align}
and \textbf{test value} function $T(Q) = \sum_{i=1}^m Z_i$. We call $Q$ a \textbf{preferable sample set} according to $\gamma_0$ if $T(Q)$ is below $(100(1-\gamma_0))$-th percentile of the distribution of $T$. 
\end{definition}

It is straightforward from the above definition that $Q$ is {preferable} according to $\gamma_0$ with probability $1-\gamma_0$. In Lemma \ref{lemmaM:good_set}, we show there exists a universal constant $C$ such that if $Q$ is {preferable} according to $\gamma_0$ then 
\begin{align*}
 \sum_{i=1}^m Z_i \le C \left(\frac{m}{\gamma_0}\right)^{1/4} \left(\frac{4}{\pi}\right)^{2d}  \left(\log\left(\frac{\pi^d}{4^{d}}\cdot \sqrt{\frac{m}{\gamma_0}}\right)\right)^{d-1}  . 
\end{align*}


\subsection{Main result}

Our main theorem in this part is stated as follows. 

{
\begin{theorem}
\label{mult:WREP_theorem}
Let $\delta, \gamma$ and $\gamma_0$ be fixed parameters with 
$0<\delta < 1$, $0 < \gamma + \gamma_0 < 1$ and $\{L_{j}\}_{{j}\in \cJ}$ be a $d$-dimensional Legendre orthonormal 
system of finite size $N = \#(\cJ)$. Denote $s_{\alpha} = {(1+\alpha)^2 s}$ and assume that 
\begin{gather}
\label{mult:complexity}
\begin{aligned}
m \ge \, & s_\alpha^2   \max\Big\{  \frac{C_{d,1}}{\delta^{12} {\gamma_0}}  \, { \log^2 \left( { \dfrac{C_{4,d} s_\alpha^{3/2}}{\delta^{2}} }\right)} \log^2(2N)   \left(\log\left(\frac{\pi^d}{4^{d}} \sqrt{\frac{m}{\gamma_0}}\right)\right)^{4d-4},   
 \\
&\qquad  \frac{C_{d,2}}{ \delta^{{15}/{2}}}  \, { \log^{3/2} \left({ \dfrac{C_{4,d} s_\alpha^{3/2}}{\delta^{2}} }\right)} \log (2N) 
 \left(\log \left( \! C{m}   \left(\! \frac{\pi}{4}\right)^{{d}}   \right) \right)^{{d-1}} ,
\\
 & \qquad \quad  \frac{   {C_{d,3}} }{ \delta^{4} }  { \log\left( { \dfrac{C_{4,d} s_\alpha^{3/2}}{\delta^{2}} }\right)}    \log\left( \frac{ \log( { {C_{4,d} s_\alpha^{3/2}}/{\delta^{2}} })}{\gamma \delta}\right) \left(  \log\left(C{\pi^{d} m}     \right) \right)^{d-1} \Big\}. 
\end{aligned}
 \end{gather}
Let $\by_1,\by_2,\ldots,\by_m$ be drawn independently from uniform distribution on $[-1,1]^d$. Then with probability exceeding $1-(\gamma + \gamma_0)$, 
the normalized sampling matrix $\bA \in \R^{m\times {N}}$ (defined as in \eqref{defA}) satisfies
\begin{align}
  \|\bm{Az}\|_2^2 > (1 -  \delta) \|\bz\|^2_2, \ \ \forall \bm{z}\in {C}(s;\alpha). \label{mult:problem1}
\end{align}
Here, $C_{d,1} = {C} \left(\dfrac{4}{\pi}\right)^{4d}  $, $C_{d,2} = \dfrac{C { (d+1)} }{{(d-1)!)^{\frac12}}} \cdot  \left( \dfrac{64\sqrt{2}}{\pi^2}\right)^d $, $C_{d,3} = \dfrac{C (d+1)}  {  (d-1)! } \cdot \left( \dfrac{4 }{\pi }\right)^{2d}$, and $C_{d,4} = \dfrac{C  (d+1)(d-1)^{d-1}  }{  (d-1)!    }\cdot \dfrac{2^{4d}}{\pi^{{3d}/{2}} }$.
\end{theorem}
}

We observe that for $d=1$, condition \eqref{mult:complexity} retrieves the estimate \eqref{spL:complexity} in the univariate case (up to a minor change of the order of $\delta$ inside log factors). The number of log factors for multi-dimensional setting however scales linearly in $d$ and is at most $4d$; therefore, the sample complexity grows exponentially in dimension. Similarly to the univariate case, random matrices associated with preferable sets are more likely to satisfy restricted eigenvalue property. Also, preferable sets according to moderate and high $\gamma_0$ reduce the required number of samples for sparse reconstructions, as shown in the following proposition. 

{
\begin{proposition}
\label{mult:WREP_theorem_pref_set}
Let $\delta, \gamma$ and $\gamma_0$ be fixed parameters in $(0,1)$ and $\{L_{j}\}_{{j}\in \cJ}$ be a $d$-dimensional Legendre orthonormal 
system of finite size $N = \#(\cJ)$.  Denote $s_{\alpha} =  {(1+\alpha)^2 s}  $ and assume that
\begin{gather}
\label{mult:complexity_2}
\begin{aligned}
m \ge \, & s_\alpha^2   \max\Big\{  \frac{C_{d,1}}{\delta^{12} {\gamma_0}}  \, { \log^2 \left( { \dfrac{C_{4,d} s_\alpha^{3/2}}{\delta^{2}} }\right)} \log^2(2N)     \left(\log\left(\frac{\pi^d}{4^{d}} \sqrt{\frac{m}{\gamma_0}}\right)\right)^{4d-4},   
 \\
&\qquad  \frac{C_{d,2}}{ \delta^{{15}/{2}}}  \, { \log^{3/2} \left({ \dfrac{C_{4,d} s_\alpha^{3/2}}{\delta^{2}} }\right)} \log (2N) 
 \left(\log \left( \! C{m}   \left(\! \frac{\pi}{4}\right)^{{d}}   \right) \right)^{{d-1}} ,
\\
 & \qquad \quad  \frac{   {C_{d,3}} }{ \delta^{4} }  { \log\left( { \dfrac{C_{4,d} s_\alpha^{3/2}}{\delta^{2}} }\right)}    \log\left( \frac{ \log( { {C_{4,d} s_\alpha^{3/2}}/{\delta^{2}} })}{(1-\gamma_0) \gamma \delta}\right) \left(  \log\left(C{\pi^{d} m}     \right) \right)^{d-1} \Big\}. 
\end{aligned}
 \end{gather}
Let $Q = \{\by_1,\by_2,\ldots,\by_m\}$ be a set of samples drawn independently from the uniform distribution on $[-1,1]^d$. If $Q$ is a preferable set according to $\gamma_0$, then with probability exceeding $1-\gamma$, 
the normalized sampling matrix $\bA \in \mathbb{R}^{m\times {N}}$ (defined as in \eqref{defA}) satisfies
\begin{align}
  \|\bm{Az}\|_2^2 > (1 -  \delta) \|\bz\|^2_2,\ \ \forall \bm{z}\in {C}(s;\alpha). \label{mult:problem1_prefset}
\end{align}
Here, $C_{d,1},C_{d,2},C_{d,3}$ and $C_{d,4}$ are defined as in Theorem \ref{mult:WREP_theorem}. 
\end{proposition}
}

We will prove Theorem \ref{mult:WREP_theorem} in Sections \ref{mult:mainproof}. Proposition \ref{mult:WREP_theorem_pref_set} can be derived directly from the proof of Theorem \ref{mult:WREP_theorem} and will be skipped. First, we need some preparatory lemmas.

\subsection{Supporting lemmas} 
{
For $d\ge 2$, we define the function
\begin{align*}
H_d(\beta) = \beta \left(\log(\frac{1}{\beta})\right)^{d-1}.
\end{align*}
In particular, the case $d =2$ gives
$H_2(\beta) = \beta \log(\frac{1}{\beta})$. Note that $\beta \mapsto \beta \log(\frac{1}{\beta})$ is strictly increasing for $0 < \beta  < 1/e$, and $\beta \mapsto H_d(\beta)$ is strictly increasing for $0< \beta < \frac{1}{e^{d-1}}$, the inverse function $K_d(M) := H_d^{-1}(M)$ is well-defined for $0< M < H_d(\frac{1}{e^{d-1}}) = (\frac{d-1}{e})^{d-1}$. We prove a basic identity involving $H_d$ and $K_d$. 

\begin{lemma}
\label{lemmaM:est}
Let $0< \beta <  (\frac{d-1}{e})^{d-1}$, there follows 
\begin{align*}
H_d(K_d^2 (\beta)) = 2^{d-1} \beta K_d (\beta). 
\end{align*}
\end{lemma}

\begin{proof}
By definition of $H_d$,
\begin{align*}
H_d(K_d^2(\beta)) & = K_d^2(\beta) \left(\log(\frac{1}{K_d^2(\beta)})\right)^{d-1}  = 2^{d-1} K_d^2(\beta) \left(\log(\frac{1}{K_d(\beta)})\right)^{d-1}
\\
& =  2^{d-1} K_d(\beta) H_d(K_d(\beta))  = 2^{d-1} \beta K_d(\beta). 
\end{align*}
\end{proof}

While an analytical formula of $K_d(M)$ is not known to us, the following lemma provides an estimate of $K_d(M)$.  
\begin{lemma}
\label{lemmaM:supp}
For all $d\ge 2$, $0<M<(\frac{d-1}{e})^{d-1}$, there holds 
\begin{align}
\label{supp:est2}
& K_{d}(M) \ge \frac{M}{(d-1)^{d-1}} \exp\left(\!-(d-1)\log^{\frac{1}{2}}\left(\frac{d-1}{M^\frac{1}{d-1}}\right)\right).
\\
\text{Consequently, }\ &\log(K_{d}(M)) \ge - \log\(\frac{(d-1)^{d-1}}{M}\) -(d-1)^{\frac12}\log^{\frac{1}{2}}\left(\frac{(d-1)^{d-1}}{M}\right). \notag
\end{align}
\end{lemma}
\begin{proof}
We first prove \eqref{supp:est2} for $d = 2$. It is enough to show that 
\begin{align}
\label{supp:est2b}
\beta \log(\frac{1}{\beta}) \le M,\ \   \text{ for } \beta = M\exp\left(-{\log^{1/2}\left(\frac1M\right)}\right). 
\end{align}
Denote $\beta_0 = \sqrt{\log(\frac{1}{M})}$. Since $0<M<1/e$, then $1<\beta_0<\infty$. We have 
$
 \beta\log(\frac{1}{\beta}) = e^{-\beta_0^2 - \beta_0}(\beta_0^2 + \beta_0)
 $ 
 and 
 $
 M = e^{-\beta_0^2}, 
$
thus 
$
\beta \log(\frac{1}{\beta}) \le M
$
is equivalent to
$
\beta_0^2 + \beta_0 \le e^{ \beta_0},\, \forall \beta_0 > 1
$. An inspection of the map $\beta_0 \mapsto e^{ \beta_0} - \beta_0^2 - \beta_0$ over $(1,\infty)$ proves the desired inequality. 

For $d>2$, from \eqref{supp:est2b} and by change of variables, one can show that 
\begin{align*}
H_d(\beta) \le  M,\ \ \text{ for } \beta  =  \frac{M}{(d-1)^{d-1}} \exp\left(\!-(d-1)\log^{\frac{1}{2}}\left(\frac{d-1}{M^\frac{1}{d-1}}\right)\right), 
\end{align*}
concluding the proof. 
\end{proof}
}


\begin{lemma}
\label{mult:volume}
Let $d\ge 1$. Define $v_d(r):= \varrho \left( \by\in \cU: \prod_{k=1}^d{(1- y_k^2)} \le r \right)$. Then, for any $0< r \le 1$, 
{
\begin{align}
 \frac{r}{2(d-1)!} \left(\log\(\frac{1}{\sqrt{r}}\)\right)^{d-1}  & <  \, v_d(r)  \le \,  \frac{ r}{(d-1)!} \left(\log\(\frac{2^{d}e}{\sqrt{r}}\)\right)^{d-1},   \label{mult:est_vol}
\\
\text{ or equivalently,}  \ \ \  \frac{H_d(r)}{2^d (d-1)!}   & <  \, v_d(r)  \le \frac{ 2^{d+1} e^2}{(d-1)!} \, H_d(\frac{r}{4^d e^2}).  \notag 
\end{align}
}
\end{lemma}
\begin{proof}
First, we prove the upper bound in \eqref{mult:est_vol} by induction on $d$. Observe that $v_d(r) = 1,\, \forall d\ge 1, r\ge 1$. Consider $r< 1$. For $d = 1$, we have 
$$
v_1(r) = \varrho \left( \by \in [-1,1]: {(1-\by^2)} \le r \right)  = 1 - \sqrt{1-r} \le r. 
$$
For $d\ge 1$, applying Fubini's theorem gives the recurrence relation 
\begin{align}
\label{mult_vol:est1}
v_{d+1}(r) &= \int_{-1}^{1} v_{d}\left(\frac{r}{1-y_{d+1}^2}\right) \frac{dy_{d+1}}{2} 
 = \int_{0}^{\sqrt{1-r}} v_{d}\left(\frac{r}{1-y_{d+1}^2}\right) {dy_{d+1}} +  \int_{\sqrt{1-r}}^1 1 {dy_{d+1}}.
\end{align}
By a change of variable and induction hypothesis, it gives
\begin{align*}
\int\limits_{0}^{\sqrt{1-r}} v_{d}\left(\frac{r}{1-y_{d+1}^2}\right) {dy_{d+1}}  = \frac{r}{2} \int\limits_r^1 \frac{v_d(u) du}{u^2 \sqrt{1 - \frac ru}} \le  \frac{ {r}}{ (d-1)!} \int\limits_r^1 \frac{  \left(\log\(\frac{2^{d}e}{\sqrt{u}} \)\right)^{d-1} du}{2\sqrt{u (u - r)}}.
\end{align*}
For all $u\in [r,1]$, we observe that 
$ 
\log\(\frac{2^{d+1}e} {\sqrt{u} + \sqrt{u-r}}  \) \ge  \log\(\frac{2^{d}e}{\sqrt{u}}\) \ge 0,
$
thus
$$
 \left( \log\(\frac{2^{d+1}e} {\sqrt{u} + \sqrt{u-r}}  \)\right)^{d-1} \ge   \left( \log\(\frac{2^{d}e}{\sqrt{u}}\) \right)^{d-1},\ \ \forall d\ge 1. 
$$
There follows 
\begin{gather}
\label{mult_vol:est2}
\begin{aligned}
&\int\limits_{0}^{\sqrt{1-r}} v_{d}\left(\frac{r}{1-y_{d+1}^2}\right) {dy_{d+1}}  \le \frac  {r}{(d-1)!} \int\limits_r^1 \frac{ \left(   \log\(\frac{2^{d+1 } e} {\sqrt{u} + \sqrt{u-r}}  \) \right)^{d-1}  du}{2 \sqrt{u (u - r)}}
\\   
&\qquad \le     \, \frac{ {r}}{d!} \left( \left(   \log\(\frac{2^{d+1} e} {\sqrt{u} + \sqrt{u-r}}  \) \right)^{d}  \right) \Biggl |_{u =1}^{r}
   \\  
&\qquad  =       \frac{ {r}}{d!} \left( \left(\log\(\frac{2^{d+1} e } {\sqrt{r}} \)  \right)^{d}   -  { \left(\log\(\frac{2^{d+1 } e }{1 + \sqrt{1-r}}\)\right)^{d} }\right)   . 
\end{aligned}
\end{gather}
Combining \eqref{mult_vol:est1} and \eqref{mult_vol:est2} and applying the estimate $1 - \sqrt{1-r} \le r$, we arrive at  
\begin{align*}
v_{d+1}(r) \le   \frac{ {r}}{d!} \left( \left(\log\(\frac{2^{d + 1} e } {\sqrt{r}} \)  \right)^{d} -  { \left(\log\(\frac{2^{d+1}e}{1 + \sqrt{1-r}}\)\right)^{d} }\right) + r. 
\end{align*}
It is enough to prove $ \left(\log\(\frac{2^{d+1}e }{1 + \sqrt{1-r}}\)\right)^{d} /d! \ge 1$. Applying Stirling's approximation $d! < \sqrt{2\pi d} \(\frac{d}{e}\)^d e^{1/(12d)} $, \cite{Robbin55}, we will show that 
$
\log\(\frac{2^{d+1}e }{1 + \sqrt{1-r}}\)\ge  ({2\pi d})^{\frac{1}{2d}} \(\frac{d}{e}\) e^{\frac{1}{12d^2}},
$
or equivalently, after some rearrangement,  
\begin{align}
\label{mult_vol:est2b}
\frac{\log\(2 \(  \frac{2e }{1 + \sqrt{1-r}}\)^{1/d}\)}{{e^{1/d}}} \ge   \left[\frac{\sqrt{2\pi } e^{1/(12d)}}{ e}\right]^{1/d}  \frac{d^{1/{(2d)}}}{e}. 
\end{align}
We observe 
\begin{align}
\label{mult_vol:est3}
\frac{\log\(2 \(  \frac{2e }{1 + \sqrt{1-r}}\)^{1/d}\)}{{e^{1/d}}} \ge \frac{\log(2e^{1/d})}{e^{1/d}} \ge \frac{\log(2e)}{e}, \ \ \forall d\ge 1,\, 0\le r \le 1. 
\end{align}
On the other hand, it is easy to check that 
\begin{align}
\label{mult_vol:est4}
\left[\frac{\sqrt{2\pi } e^{1/(12d)}}{ e}\right]^{1/d}  \frac{d^{1/{(2d)}}}{e} \le \frac{\sqrt{2\pi } e^{1/12}}{ e} . \frac{e^{1/(2e)}}{e} . 
\end{align}
A combination of \eqref{mult_vol:est3} and \eqref{mult_vol:est4} gives \eqref{mult_vol:est2b}.

Next, we establish the lower bound of $v_d(r)$, again by induction on $d$. For $d = 1$, 
$
v_1(r)  = 1 - \sqrt{1-r} > \frac{r}{2}. 
$
For $d\ge 1$, recall 
$$
v_{d+1}(r)  = \int\limits_{0}^{\sqrt{1-r}} v_{d}\left(\frac{r}{1-y_{d+1}^2}\right) {dy_{d+1}} +  \int\limits_{\sqrt{1-r}}^1 1 {dy_{d+1}} > \int\limits_{0}^{\sqrt{1-r}} v_{d}\left(\frac{r}{1-y_{d+1}^2}\right) {dy_{d+1}}.
$$ 
By induction hypothesis, it gives
\begin{align*}
v_{d+1}(r) &\,> \frac{r}{2} \int\limits_r^1 \frac{v_d(u) du}{u^2 \sqrt{1 - \frac ru}} > \frac{r}{4(d-1)!} \int\limits_r^1 \frac{ \left(\log\(\frac{1}{\sqrt{u}}\)\right)^{d-1}  du}{u \sqrt{1 - \frac ru}}  
\\
& \, >  \frac{r}{4(d-1)!} \int\limits_r^1 \frac{ \left(\log\(\frac{1}{\sqrt{u}}\)\right)^{d-1}  du}{u}  = \frac{r}{2.d!} \left(\log\(\frac{1}{\sqrt{r}}\)\right)^{d} , 
\end{align*}
as desired.
\end{proof}

We define 
$
\cS_{\mu}: = \Omega^{-1}([\mu,\infty)) = \{ \by\in \cU: |\Omega(\by)| \ge \mu\}.
$
With notice that 
$
 \varrho\left(\cS_{\mu} \right) = v_d\left( \dfrac{2^{4d}}{\pi^{2d}\mu^4} \right), 
$
the following result is a direct consequence of Lemma \ref{mult:volume} and a multivariate version of Lemma \ref{lemma:prob_tail}.  
\begin{corollary}
\label{mult:prob_tail}
For $d\ge 2$, $\mu \ge \dfrac{2^d}{{\pi}^{d/2}}$, there holds
\begin{gather}
\label{mult_vol:est5}  
\begin{aligned}
 {\frac{1}{2^d (d-1)!}\, H_d( \frac{2^{4d}}{\pi^{2d}\mu^4}) } &< \varrho\left( \cS_{\mu} \right)  
\le  { \frac{2^{d+1} e^2}{(d-1)!} \, H_d\left({\frac{2^{2d}}{e^2 \pi^{2d}\mu^4}}\right)}.  
 \end{aligned}
\end{gather}
\end{corollary}

We proceed to generalize estimates in Lemmas \ref{lemma:end_set} and \ref{lemma:imp_integral1} for multi-dimensional setting. 


\begin{lemma}
\label{mult:end_set}
For $d\ge 2$, $ \mu \ge \dfrac{2^d}{{\pi}^{d/2}}$, there holds 
\begin{align}
\label{mult_vol:est6}
\int_{\cS_{\mu}}  |\Omega(\by)|^2 d\varrho \le\,   {\frac{2^{4d}(d+1)}{ \pi^{2d} (d-1)! }\cdot \frac{\(\log\({e} {\pi^{d}\mu^2} \)\)^{d-1}}{ \mu^2} =  \frac{2^{4d}e(d+1)}{ \pi^{d} (d-1)! }\ H_d(\frac{1}{e\pi^d \mu^2}) . }   
\end{align}
\end{lemma}

\begin{proof}
We have 
\begin{align*}
\int_{\cS_{\mu}}  |\Omega(\by)|^2 d\varrho & = \int_{0}^{\infty} \varrho(\by\in \cS_{\mu}: |\Omega(\by)|^2 \ge t) dt = \int_0^{\infty}\varrho (\cS_{\mu} \cap \cS_{\sqrt{t}}) dt
\\
& = \int_0^{\mu^2} \varrho (\cS_{\mu} ) dt + \int_{\mu^2}^{\infty} \varrho (\cS_{\sqrt{t}} ) dt. 
\end{align*}
Applying the upper estimates of $\varrho(\cS_{\mu})$ and $\varrho(\cS_{\sqrt{t}})$ in Corollary \ref{mult:prob_tail}, there follows 
\begin{align}
\int_{\cS_{\mu}}  & |\Omega(\by)|^2 d\varrho  \le   \frac{2^{4d}}{ \pi^{2d} (d-1)! } \left(  \frac{\left(\log\({{e} {\pi^{d}\mu^2}} \)\right)^{d-1}}{ \mu^2}    + \int_{\mu^2}^{\infty} \frac{\left(\log\({{e} {\pi^{d}t}} \)\right)^{d-1}}{ t^2}   dt \right) \notag
\\
& = \frac{2^{4d}}{ \pi^{2d} (d-1)! } \left(  \frac{\left(\log\({{e} {\pi^{d}\mu^2}} \)\right)^{d-1}}{ \mu^2}    + \sum_{k=0}^{d-1} \frac{(d-1)!}{k!}\cdot \frac{\left(\log\({{e} {\pi^{d}\mu^2} }\)\right)^{k}}{ \mu^2}    \right). \label{mult_vol:est7}
\end{align}
For $\mu \ge 1$, observe that $d< \log(e\pi^d \mu^2) $, thus $\frac{(d-1)!}{k!} \le \left(\log\({e} {\pi^{d}\mu^2} \)\right)^{d-1-k},\, \forall 0\le k \le d-1$. Substituting these estimates to \eqref{mult_vol:est7} yields \eqref{mult_vol:est6}.    
\end{proof}

\begin{lemma}
\label{lemma:mult_integral}
Let $d\ge 2$, there exists a universal constant $C > 0$ such that for all $\beta \ge 10$ and $\tau \in \{0,\frac14,\frac12\}$,  
\begin{align}
 \label{eq:mult_integral}
 & \int_{\cU} \frac{\exp\( - {\beta}\prod_{k=1}^d{(1-y_k^2)^{1/2}}\)}{\prod_{k=1}^d{(1-y_k^2)^{\tau}}} d\varrho
 \le C \,  {\beta^{2\tau -2}}(\log(\beta))^{d-1} .
\end{align}
\end{lemma}
\begin{proof}
Denote $\cB_r = [-r,r]^{d-1}$ for $0<r < 1$, and represent $\varrho$ as $\varrho = \prod_{k=1}^d \varrho_k$, where $\varrho_k$ is the univariate uniform measure on $[-1,1]$. First, for any $r\in (0,1)$, 
{
\allowdisplaybreaks
\begin{align*}
&\int_{\cU} \frac{\exp\( - {\beta}\prod_{k=1}^d{(1-y_k^2)^{1/2}}\)}{\prod_{k=1}^d{(1-y_k^2)^{\tau}}} d\varrho 
\\
 &\quad = \int_{\cB_r\, \cup \, (\cB_1 \setminus \cB_r)} \left( \int_{-1}^1  \frac{\exp\( - {\beta}\prod_{k=1}^d{(1-y_k^2)^{1/2}}\)}{\prod_{k=1}^d{(1-y_k^2)^{\tau}}}  d\varrho_1 \right) d\varrho_2\ldots d\varrho_{d}
 \\
& \quad = \int_{\cB_r}\prod_{k=2}^d{(1-y_k^2)^{- \tau}} \left( \int_{-1}^1 \frac{\exp\( - {\beta}\prod_{k=1}^d{(1-y_k^2)^{1/2}}\)}{(1-y_1^2)^{\tau}} d\varrho_1 \right) d\varrho_2\ldots d\varrho_{d}
\\
& \qquad  + \int_{\cB_1 \setminus \cB_r}\prod_{k=2}^d{(1-y_k^2)^{- \tau}} \left( \int_{-1}^1 \frac{\exp\( - {\beta}\prod_{k=1}^d{(1-y_k^2)^{1/2}}\)}{(1-y_1^2)^{\tau}} d\varrho_1 \right) d\varrho_2\ldots d\varrho_{d}
\\
& \quad \le C \beta^{2\tau-2} \int_{ \cB_r}\prod_{k=2}^d{(1-y_k^2)^{- 1}}  d\varrho_2\ldots d\varrho_{d} 
\\
& \qquad +  \int_{-1}^1{(1-y_1^2)^{-\tau}} d\varrho_1   \int_{\cB_1 \setminus \cB_r}\prod_{k=2}^d{(1-y_k^2)^{- \tau}} d\varrho_2\ldots d\varrho_{d}
\\
& \quad = C \beta^{2\tau-2} \left(\int_{0}^r \frac{dy}{1-y^2}\right)^{d-1} + \int_0^1 \frac{dy}{(1-y^2)^{\tau}} 
\\
& \qquad\qquad \qquad\qquad\qquad \times \left[ \left(\int_0^1 \frac{dy}{(1-y^2)^{\tau}} \right)^{d-1} -  \left( \int_0^r \frac{dy}{(1-y^2)^{\tau}}\right)^{d-1} \right] 
\\
& \quad \le C \beta^{2\tau-2} \left(\int_{0}^r \frac{dy}{1-y^2}\right)^{d-1} + (d-1) \left(\int_0^1 \frac{dy}{(1-y^2)^{\tau}} \right)^{d-1}  \int_r^1 \frac{dy}{(1-y^2)^{\tau}}, 
\end{align*}
}
where the first inequality is derived from Lemma \ref{lemma:imp_integral1} and $\exp\( - {\beta}\prod\limits_{k=1}^d{(1-y_k^2)^{\frac{1}{2}}}\) \le 1$. 
Define the constants $C_\tau = 1,\, _2F_1(\frac{1}{4},\frac{1}{2};\frac{3}{2};1)\  (\simeq 1.2)$ and $\frac{\pi}{2}$, respectively for $\tau = 0,\, \frac14$, and $\frac12 $, where $_2F_1$ is the hypergeometric function, observe that 
\begin{align*}
&\int_{0}^r \frac{dy}{1-y^2}  = \frac{1}{2} \log\(\frac{1+y}{1-y}\)\biggl |_{y=0}^r = \frac{1}{2} \log\(\frac{1+r}{1-r}\),
\\
& \int_{r}^1 \frac{dy}{(1-y^2)^{\tau}} =
\begin{cases}
1 - r,\ &\text{ if }\tau = 0,
\\
(y\, _2F_1(\frac14,\frac12;\frac32;y^2))|_{y=r}^1 {\, \le C_{\tau} (1-r)^{3/4}},\ &\text{ if }\tau= 1/4,
\\
\pi/2 - \arcsin(r) {\, \le \frac{\pi}{2}\sqrt{1-r}},\ &\text{ if }\tau = {1/2}.\ 
\end{cases}
\end{align*}
Therefore, for $\tau \in \{0,1/4,1/2\}$ and $0<r < 1$, we have
\begin{gather}
\label{mul:est1}
\begin{aligned}
&\int_{\cU} \frac{\exp\( - {\beta}\prod_{k=1}^d{(1-y_k^2)^{1/2}}\)}{\prod_{k=1}^d{(1-y_k^2)^{\tau}}} d\varrho 
\\
&\qquad \qquad \qquad \le  C \frac{\beta^{2\tau -2}}{2^{d-1}}\(\log\(\frac{1+r}{1-r}\)\)^{d-1} +  (d-1) C_{\tau} ^{d}  (1-r)^{1-\tau}. 
\end{aligned}
\end{gather}
Choose $r= 1 - \frac{2}{\beta^2 }  \in (0,1)$, then 
\begin{align}
\label{mul:est2}
(1-r)^{1-\tau} = {2^{1-\tau}\beta^{2\tau -2}  },\ \ \ \log\(\frac{1+r}{1-r}\) \le \log\(\frac{2}{1-r}\) = 2\log(\beta). 
\end{align}   
Substituting \eqref{mul:est2} to \eqref{mul:est1} yields 
\begin{align*}
& \int_{\cU} \frac{\exp\( - {\beta}\prod_{k=1}^d{(1-y_k^2)^{1/2}}\)}{\prod_{k=1}^d{(1-y_k^2)^{\tau}}} d\varrho 
\le  C {\beta^{2\tau -2}}\(\log(\beta)\)^{d-1} +  (d-1)2^{1-\tau}  C_{\tau} ^{d} \beta^{2\tau -2}  \notag
\\
\le\, &  C {\beta^{2\tau -2}}\(\log(\beta)\)^{d-1} +  2^{1-\tau} C_{\tau} ((d-1)^{\frac{1}{d-1}} C_{\tau} )^{d-1} \beta^{2\tau -2}
\le\,   C\,  {\beta^{2\tau -2}}(\log(\beta))^{d-1} , 
\end{align*}
as desired, where the last inequality comes from observation $(d-1)^{\frac{1}{d-1}} C_{\tau} < \log(10) \le \log(\beta)$. 
\end{proof}

The next lemma establishes that preferable sets according to $\gamma_0$ have their test values bounded from above by $C \left(\dfrac{m}{\gamma_0}\right)^{1/4} \left(\dfrac{4}{\pi}\right)^{2d}  \left(\log\left(\dfrac{\pi^d}{4^{d}}\cdot \sqrt{\dfrac{m}{\gamma_0}}\right)\right)^{d-1}$.

\begin{lemma}
\label{lemmaM:good_set}
For $m>0$, $\gamma_0 \in (0,1)$ satisfying $\gamma_0 \le \frac{m}{20^2} \left(\frac{\pi}{4}\right)^{2d}$, let $\by_1,\by_2,\ldots,\by_m$ be sampling points drawn independently from the uniform distribution on $[-1,1]^d$ and  
$$
Z_i =  { \Omega(\by_i)} \exp\(-\dfrac{1}{2 \Omega^2(\by_i)}  \sqrt{\dfrac{m}{\gamma_0}} \),\quad \forall i \in \{1,\ldots,m\}. 
$$
be $m$ i.i.d random variables. With probability exceeding $1-\gamma_0$, there holds 
\begin{align}
\label{mult:est4}
 \sum_{i=1}^m Z_i \le  C \left(\frac{m}{\gamma_0}\right)^{1/4} \left(\frac{4}{\pi}\right)^{2d}  \left(\log\left(\frac{\pi^d}{4^{d}}\cdot \sqrt{\frac{m}{\gamma_0}}\right)\right)^{d-1} . 
 \end{align}
\end{lemma}
\begin{proof}
By Lemma \ref{lemma:mult_integral}, 
\begin{gather}
\label{mult:est5}
\begin{aligned}
&\E Z_i \le \, C \left(\frac{4}{\pi}\right)^{2d} \left(\frac{\gamma_0}{m}\right)^{3/4} \left(\log\left(\frac{\pi^d}{2^{2d+1}}\cdot \sqrt{\frac{m}{\gamma_0}}\right)\right)^{d-1},
\\
&\text{Var}(Z_i) \le \E |Z_i|^2 \le  C \left(\frac{4}{\pi}\right)^{2d} \left(\frac{\gamma_0}{m}\right)^{1/2} \left(\log\left(\frac{\pi^d}{2^{2d}}\cdot \sqrt{\frac{m}{\gamma_0}}\right)\right)^{d-1}. 
\end{aligned}
\end{gather}
Referring $\E Z_i$ and $\text{Var}(Z_i)$ as $\E Z$ and $\text{Var}(Z)$, we apply Chebyshev inequality to obtain with probability exceeding $1-\gamma_0$: 
\begin{gather}
\label{mult:est6}
\begin{aligned}
 \sum_{i=1}^m Z_i  & \le m \E Z  +    \sqrt{\frac{m}{\gamma_0}} (\E |Z|^2)^{1/2}.
\end{aligned}
\end{gather}
Substituting \eqref{mult:est5} to \eqref{mult:est6} yields 
\begin{align*}
 \sum_{i=1}^m Z_i  & \le C\, m^{1/4} \gamma_0^{3/4 }  \left(\frac{4}{\pi}\right)^{2d}  \left(\log\left(\frac{\pi^d}{4^{d}}\cdot \sqrt{\frac{m}{\gamma_0}}\right)\right)^{d-1}  \notag 
 \\
 & \qquad\qquad \qquad +     C \left(\frac{m}{\gamma_0}\right)^{1/4} \left(\frac{4}{\pi}\right)^{d}  \left(\log\left(\frac{\pi^d}{4^{d}}\cdot \sqrt{\frac{m}{\gamma_0}}\right)\right)^{\frac{d-1}{2}} 
 \\
 & \le   C \left(\frac{m}{\gamma_0}\right)^{1/4} \left(\frac{4}{\pi}\right)^{2d}  \left(\log\left(\frac{\pi^d}{4^{d}}\cdot \sqrt{\frac{m}{\gamma_0}}\right)\right)^{d-1} . 
\end{align*}
\end{proof}

Recall $
\cE_{s,\alpha} :=  \partial B_2 \cap C(s;\alpha),
$
our last lemma in this section involves an estimate of covering number for $\cE_{s,\alpha}$. This is an extension of Lemma \ref{note:lemma4} to the setting where $(\cU,\varrho)$ is multi-dimensional sample space. 

\begin{lemma}
\label{mult:lemma4}
For { $0 < \varsigma < 1$}, $\mu >0$, $m>0$ and $0 < \gamma_0 < \min\{1, \frac{m}{20^2} \left(\frac{\pi}{4}\right)^{2d}\}$, there 
exists a set $D \subset \R^N$ such that the following hold:
\begin{enumerate}
\item[(i)] If $Q = \{\by_1,\ldots,\by_m\}$ is a set of $m$ i.i.d. samples drawn from $(\cU,\varrho)$, preferable according to $\gamma_0$, then for all $\bz\in \cE_{s,\alpha}$, there exists $\bz' \in D$ depending on $\bz$ and $Q$ satisfying 
\begin{align}
&|\psi(\by, \bm{z} - \bm{z}')| \le \mu \quad\mbox{for all }\by \text{  in a subset }   \cU^\star \text{ of } \cU\mbox{ with }\varrho(\cU^\star) > { 1-\varsigma} ,  \label{mult:metric1}
\\
& |\psi(\by_i, \bm{z} - \bm{z}')| \le \mu\quad \mbox{for all }\by_i \in Q. \label{mult:metric2}
\end{align}

\item[(ii)] The cardinality of $D$ satisfies 
{
\begin{gather}
\label{mult:card}
\begin{aligned}
\log(\#(D)) \le\,  & C \frac{ (1+\alpha)^2 s}{\mu^2} \log(2N) \max\biggl\{\left(\frac{4}{\pi}\right)^d \left[ K_d(\frac{2^{d-1}\varsigma}{C}) \right]^{-\frac12} , 
\\
& \qquad \qquad \qquad \sqrt{\frac{m}{\gamma_0}} \left(\frac{4}{\pi}\right)^{2d}  \left(\log\left(\frac{\pi^d}{4^{d}}\cdot \sqrt{\frac{m}{\gamma_0}}\right)\right)^{2d-2}\biggl\} .
\end{aligned}
\end{gather}
 }
 
\end{enumerate}
\end{lemma} 

\begin{proof}
The proof follows closely that of Lemma \ref{note:lemma4}, with all estimates involving the univariate envelope bound of Legendre polynomials replaced by those with the multivariate bound. For brevity, we only show the critical changes and refer the readers to Lemma \ref{note:lemma4} for the full arguments. The first change occurs in \eqref{est7b} where we now  have by Lemma \ref{lemma:mult_integral}  
\begin{align*}
 \int_D \left( \int_\cU \chi(\by,\overline\bz) d\varrho \right) d\overline{\lambda} & =  \int_\cU \left( \int_D \chi(\by,\overline\bz)d\overline{\lambda} \right) d\varrho
\le  \int_\cU  2 \exp\({-\frac {M \mu^2}{ 2(1+\alpha)^2  \Omega^2(\by) s }}\)    d\varrho 
\\
& \le    C \beta^{-2 }\left(\log\left(\beta\right)\right)^{d-1}. 
\end{align*}
Here, $\beta = \dfrac{M\mu^2}{2(1+\alpha)^2s}\left(\dfrac{\pi}{4}\right)^d$. Applying Markov's inequality, with probability exceeding 2/3, $\overline\bz\in D$ satisfies
\begin{align}
\label{mult:est7}
 \P_{\by}&\(\left| \psi(\by,\overline\bz - \bm{z}) \right| \!\ge\! \mu \) = 
 \int_\cU \chi(\by,\overline\bz) d\varrho  
 \le\,   C \beta^{-2 }\left(\log\left(\beta\right)\right)^{d-1} = { \frac{ C}{2^{d-1}} H_d(\frac{1}{\beta^2})}. 
\end{align}
From the definition of $K_d$, one has $\P_{\by} \(\left| \psi(\by,\overline\bz - \bm{z}) \right| \ge \mu \) \le \varsigma $, assuming the trivial condition { ${\varsigma} < C \(\dfrac{d-1}{2 e}\)^{d-1} $ and

\begin{align}
\label{mult:est8}
\beta \ge \left[ K_d\(\frac{2^{d-1}\varsigma}{C}\)\right]^{-\frac12}  ,\ \text{or equivalently,}\ \ M \ge \dfrac{ 2 (1+\alpha)^2  s}{\mu^2} \left(\frac{4}{\pi}\right)^d   \left[K_d \(\frac{2^{d-1}\varsigma}{C}\)\right]^{-\frac12}. 
\end{align}
}

Next, consider a preferable sample set $Q = \{\by_1,\ldots,\by_m\}$ according to $\gamma_0$. Similar to Lemma \ref{note:lemma4}, we have 
\begin{gather}
 \label{mult:est9}
\begin{aligned}
& \E_\epsilon \max_{\by_i\in Q} \left| \sum_{k=1}^M \epsilon_k  \psi(\by_i,\bz_k)  \right|  \le  \,  (1+\alpha) \sqrt{Ms} \sqrt[4]{\dfrac{m}{\gamma_0}} \, 
\\
& \qquad\qquad\qquad\qquad  + \,   \sqrt{2 \pi Ms} (1+ \alpha) \sum_{i=1}^m { \Omega(\by_i)} \exp\(-\frac{1}{2 \Omega^2(\by_i)}  \sqrt{\dfrac{m}{\gamma_0}} \).
\end{aligned}
\end{gather}
Applying Lemma \ref{lemmaM:good_set} gives
\begin{align*}
&\E_\epsilon \max_{\by_i\in Q} \left| \sum_{k=1}^M \epsilon_k  \psi(\by_i,\bz_k)  \right| \le  \,  C(1+\alpha) \sqrt{ Ms}  \sqrt[4]{\dfrac{m}{\gamma_0}} \left(\frac{4}{\pi}\right)^{d}  \left(\log\left(\frac{\pi^d}{4^{d}}\cdot \sqrt{\frac{m}{\gamma_0}}\right)\right)^{d-1} .
\end{align*}
Since $\E_{\overline\bz}\(\max_{\by_i\in Q} |\psi(\by_i,\bz-\overline\bz )| \) \le \frac{2}{M} \E_{\overline\bz} \E_\epsilon \max_{\by_i\in Q} \left| \sum_{k=1}^M \epsilon_k  \psi(\by_i,\bz_k)  \right|$, applying Markov's inequality, with probability exceeding $2/3$, $\overline \bz \in D$ satisfies
 $$
\max_{\by_i\in Q} |\psi(\by_i,\bz - \overline\bz)| \le C(1+\alpha) \sqrt{ \frac sM}  \sqrt[4]{\dfrac{m}{\gamma_0}} \left(\frac{4}{\pi}\right)^{d}  \left(\log\left(\frac{\pi^d}{4^{d}}\cdot \sqrt{\frac{m}{\gamma_0}}\right)\right)^{d-1}. 
$$ 
Thus, $\max_{\by_i\in Q} |\psi(\by_i,\bz - \overline\bz)| \le \mu $ provided that 
\begin{align}
\label{mult:est10}
M \ge  \dfrac{C (1+\alpha)^2 s\sqrt{m}}{\mu^2\sqrt{\gamma_0}}  \left(\frac{4}{\pi}\right)^{2d}  \left(\log\left(\frac{\pi^d}{4^{d}}\cdot \sqrt{\frac{m}{\gamma_0}}\right)\right)^{2d-2}. 
\end{align}

Finally, for $M$ satisfying \eqref{mult:est8} and \eqref{mult:est10}, there exists $\overline{\bz}\in D$ that fulfills both \eqref{mult:metric1} and \eqref{mult:metric2}. As $\log(\#(D))\le M\log(2N)$, the proof is concluded. 
\end{proof}

\subsection{Proof of Theorem \ref{mult:WREP_theorem}}
\label{mult:mainproof}

We conduct the analysis under assumption $\delta < 1/13$; this parameter will be rescaled at the end. The proof follows the same logic to that of Theorem \ref{spL:WREP_theorem}. For brevity, we only elaborate major changes specific to multi-dimensional setting and refer the readers to Theorem \ref{spL:WREP_theorem} for the rest of the arguments. 

 For {$m\in \mathbb{N}$}, let $Q = \{\by_1,\dots,\by_m\}$ be a set of samples where $\by_1,\ldots,\by_m$ are drawn independently from the probability space $(\cU,\varrho)$. Let $\Xi \in \mathbb{R}$ whose value will be set later to meet our need, we define the set of integers 
\begin{align} 
\cL =  \Z\cap \(\frac{\log(\delta)}{\log(1+\delta)}+1, 
\dfrac{\log(\Xi)}{\log(1+\delta)}+1\).
\label{mult:defineIndices}
\end{align}

\textbf{Step 1:} under the condition that $Q$ is a preferable sample set according to $\gamma_0$, by following the same procedure in Step 1, Section \ref{subsec:mainproof_1d}, with an application of Lemma \ref{mult:lemma4} instead of Lemma \ref{note:lemma4}, we can construct $\tilde{\psi}$ approximating $\psi$ that satisfies exactly three properties in the aforementioned step, except that the cardinality of class $F_l$ is now 
\begin{align}
\log(\#F_l) \le C \frac{(1+\alpha)^2 s}{\delta^3 (1 + \delta)^{2l-2} }   & \log(2N) \max\biggl\{ \notag \left(\frac{4}{\pi}\right)^d \left[K_d \left(\frac{2^{d-1}}{C}. \frac{ \varsigma {\log(1+\delta)}  }{ \log( \Xi / \delta)} \right)\right]^{-\frac12},  \notag
\\
& \sqrt{\frac{m}{\gamma_0}}  \left(\frac{4}{\pi}\right)^{2d}  \left(\log\left(\frac{\pi^d}{4^{d}}\cdot \sqrt{\frac{m}{\gamma_0}}\right)\right)^{2d-2}\biggl\}  .   \label{mult:eq2b}
\end{align}
Partitioning $\cU$ into $I,\overline{I},\underline{I}$ and $\cU'$ as in Section \ref{subsec:mainproof_1d}, then 
\begin{align} 
&\left(1-{3\delta}/{2}\right)\tilde{\psi}(\by,\bm{z})  < |{\psi}(\by,\bm{z})| <  \left(1+ {3\delta}/{2}\right)\tilde{\psi}(\by,\bm{z}) ,\ & \forall \by \in I, \label{mult:partition1}
\\
& 0\le |{\psi}(\by,\bm{z})| < 6\delta/5 \ \ \mbox{ and }\ \ \tilde{\psi}(\by,\bm{z}) = 0 ,\ & \forall \by \in \underline{I}, \label{mult:partition2}
\\
&  |\psi(\by,\bz)| > \Xi  \qquad \quad \ \  \ \mbox{ and }\ \ \tilde{\psi}(\by,\bm{z}) = 0 ,\ & \forall \by \in \overline{I}, \label{mult:partition2b}
\\
& \varrho(\cU' )\le \varsigma,\qquad\qquad \ \ \ \mbox{ and }\ \  Q\cap \cU' = \emptyset. &  \label{mult:partition3}
\end{align}
Let $\nu = \dfrac{\Xi}{(1+\alpha)\sqrt{s}}$, note that for all $\by \in \overline{I}$, $|\Omega(\by)| \ge \dfrac{|\psi(\by,\bz)|}{(1+\alpha)\sqrt{s}} > \nu$, thus by Corollary \ref{mult:prob_tail}, $\varrho(\overline{I}) \le\!  { \dfrac{2^{d+1} e^2}{(d-1)!} \, H_d\left({\dfrac{2^{2d}(1+\alpha)^4{s}^2}{e^2 \pi^{2d}\Xi^4}}\right)}$, assuming that $  \dfrac{\Xi}{(1+\alpha)\sqrt{s}}  \ge \(\dfrac{4}{\pi}\)^{d/2}$. 

\textbf{Step 2:} We derive estimates of $\|\bA \bz\|_2$ and $\| \bz\|_2$ in terms of $\tilde{\psi}(\cdot,\bz)$. Similarly to \eqref{spL:est6} and \eqref{spL:est7}, 
 \begin{align}
\|\bm{Az}\|^2_2 & \ge {(1-{3\delta})} \sum_{i=1}^m \frac {| \tilde{\psi}(\by_i,{\bm z})|^2}{m} . \label{mult:est13}
\end{align}

For $\|\bz\|_2$, we decompose
\begin{align}
\label{mult:est13b}
\|\bz\|_2^2 = \int_{I}\! |\psi(\by,\bz)|^2 d\varrho + \int_{\underline{I}}\! |\psi(\by,\bz)|^2 d\varrho + \int_{\overline{I}} \! |\psi(\by,\bz)|^2 d\varrho + \int_{\cU'} |\psi(\by,\bz)|^2 d\varrho. 
\end{align}

The first two RHS terms can be bounded in the same way as in \eqref{spL:est9} and \eqref{spL:est10}. For the third term, from Lemma \ref{mult:end_set} and the fact that $\overline{I}\subset \mathcal{S}_{\nu}$, 
\begin{gather}
\begin{aligned}
 \label{mult:est15b}
\int_{\overline{I}} |\psi(\by,\bz)|^2 d\varrho & \le (1+\alpha)^2 s  \int_{\mathcal{S}_{\nu}} |\Omega(\by)|^2 d\varrho 
\\
&  { \le  (1+\alpha)^2 s . \frac{2^{4d}e(d+1)}{ \pi^{d} (d-1)! } H_d\left(\frac{(1+\alpha)^2 s}{e\pi^d \Xi^2 }\right) 
 =  \frac{11\delta}{12} },
\end{aligned}
\end{gather}
if we set {$\Xi = \dfrac{(1+\alpha) \sqrt{s}}{\pi^{d/2} \sqrt{e} } \left[ K_d\left(\dfrac{11\delta}{12 (1+\alpha)^2 s}.\dfrac{  \pi^{d} (d-1)! }{2^{4d}e(d+1)}\right) \right]^{-\frac12} $}. Note that for $\Xi$ and $K_d$ to be well-defined, we need 
$ \dfrac{11\delta}{12 (1+\alpha)^2 s}.\dfrac{  \pi^{d} (d-1)! }{2^{4d}e(d+1)} < \(\dfrac{d-1}{e}\)^{d-1}$. It can be checked that this condition is trivial. Also, $\Xi$ defined above satisfies the condition $  \dfrac{\Xi}{(1+\alpha)\sqrt{s}}  \ge \(\dfrac{4}{\pi}\)^{d/2}$ posed in Step 1. 

For the last term, we have 
\begin{align}
\label{mult:est16}
\int_{\cU'} |\psi(\by,\bz)|^2 d\varrho  \le (1+\alpha)^2 s  \int_{\cU'} |\Omega(\by)|^2 d\varrho \le  { \frac{11\delta}{12}}.
\end{align}
if setting {$\varsigma  := \varrho(\cS_{\nu})$. From Corollary \ref{mult:prob_tail}, note that 
\begin{align}
 \varsigma &  >  \dfrac{1}{2^d (d-1)!} \, H_d\left(2^{4d} e^2 K_d^{2}\left(\dfrac{11\delta}{12 (1+\alpha)^2 s}.\dfrac{  \pi^{d} (d-1)! }{2^{4d}e(d+1)}\right)\right). \label{mult:est16b} 
\end{align}
}
Combining all above estimates yields
\begin{gather}
\label{mult:est17}
\begin{aligned}
&\|\bz\|_2^2 - \|\bA \bz\|_2^2  \le   \ \frac{34\delta}{3} +  (1 -  3\delta) \! \left( 
 \int_{\cU}\! |\tilde \psi(\by,\bz)|^2 d\varrho - \sum_{i=1}^m \frac{| \tilde{\psi}(\by_i,{\bm z})|^2}{m} \right).   
\end{aligned}
\end{gather}

\textbf{Step 3:} we derive a positive upper bound of $ \|\bz\|_2^2 - \|\bA\bz\|_2^2 $. Let $(\kappa_l)_{l\in\mathcal{L}}$ be a sequence of positive numbers. From \eqref{mult:est17}, analogously to Section \ref{subsec:mainproof_1d}, we can prove that with probability exceeding $1 - \gamma_0 - \sum_{l\in \mathcal{L}} \exp\left(- \frac{m \kappa_l \delta}{16e}\right) (\# F_l) $, $Q$ is preferable according to $\gamma_0$ \textit{and} for all $\bz\in \cE_{s,\alpha}$, 
\begin{gather*}
\begin{aligned}
\|\bz\|_2^2 -  \|\bA\bz\|_2^2 
 \le   
\ \frac{34\delta}{3} + \delta (1- 3\delta)\int_{\cU} |\tilde \psi(\by,\bz)|^2 d\varrho 
+ 
 (1-3\delta)\sum_{l\in\cL}  
(1+\delta)^{2l}
  \kappa_l
 .
\end{aligned}
\end{gather*}
Note that $\tilde\psi(\by,\bz) = 0,\, \forall \by\notin {I}\cup \cU'$ and $\tilde\psi(\by,\bz) \le (1+\delta)\Xi,\, \forall \by\in \cU$, combining with \eqref{mult:est15b}, we estimate 
{
\begin{align*}
& \int_{\cU} |\tilde \psi(\by,\bz)|^2 d\varrho = \int_{I} |\tilde \psi(\by,\bz)|^2 d\varrho + \int_{\cU'} |\tilde \psi(\by,\bz)|^2 d\varrho 
\\
 \le\ & (1+4\delta) \int_{I} | \psi(\by,\bz)|^2 d\varrho + \int_{\cS_\nu} (1+\delta)^2 \Xi^2 d\varrho 
 \\
 <\  &   1+ 4 \delta + (1 + \delta)^2 (1+\alpha)^2 s \int_{\cS_\nu} |\Omega(\by)|^2 d\varrho < 1 + \frac{56\delta}{11}, 
\end{align*}
}
which yields 
\begin{align*}
\|\bz\|_2^2 -  \|\bA\bz\|_2^2 < 
\frac{25\delta}{2} +  
 (1-3\delta)\sum_{l\in\cL}  
(1+\delta)^{2l}
  \kappa_l.
 \end{align*}

To obtain Theorem \ref{mult:WREP_theorem}, we 
derive 
conditions on $m$ such that
\begin{align}
\label{mult:eq9b}
m & \ge  C~(\#\cL) \frac{(1+\delta)^{2l}}{\delta^2} \left[ \log(\# F_l)  
+ \log\left(\frac{\#\cL}{\gamma}\right)\right],\ \quad\quad l\in \cL. 
\end{align}
We have in view 
of \eqref{mult:eq2b} and assumption $\delta < 1/13$ that

\begin{align}
&C\,  (\#\cL) \frac{(1+\delta)^{2l}}{\delta^2} \log(\# F_l)  \notag
\\
& \qquad \le C\delta^{-6}  \, {(1+\alpha)^2 s}  \, { \log( \Xi / \delta)} \log(2N) \max\biggl\{ \notag \left(\frac{4}{\pi}\right)^d \left[ K_d \left(\frac{2^{d-1}}{C}. \frac{ \varsigma {\log(1+\delta)}  }{ \log( \Xi / \delta)} \right) \right]^{-\frac12},  \notag
\\
&\qquad \qquad \qquad \qquad \qquad \qquad \qquad \sqrt{\frac{m}{\gamma_0}}  \left(\frac{4}{\pi}\right)^{2d}  \left(\log\left(\frac{\pi^d}{4^{d}}\cdot \sqrt{\frac{m}{\gamma_0}}\right)\right)^{2d-2}\biggl\},     \label{mult:eq10}
\\
&C\,  (\#\cL) \frac{(1+\delta)^{2l}}{\delta^2}  \log\left(\frac{\#\cL}{\gamma}\right) \le   C \frac{ \log( \Xi / \delta)}{\log(1+\delta)} \cdot  \frac{(1+\delta)^{2} \Xi^2 }{\delta^2}   \cdot   \log\left( \frac{ \log( \Xi / \delta)}{\gamma \log(1+\delta)}\right)  \notag
 \\
&\  \le C \delta^{-3} \pi^{-d}    {(1+\alpha)^2 {s}} \left [K_d \left(\dfrac{11\delta}{12 (1+\alpha)^2 s}.\dfrac{  \pi^{d} (d-1)! }{2^{4d}e(d+1)}\right) \right ]^{-1}\!  { \log( \Xi / \delta)}    \log\left( \frac{ \log( \Xi / \delta)}{\gamma \delta}\right).  \notag 
\end{align} 

To fulfill \eqref{mult:eq9b}, we set $m$ greater than the right hand sides in \eqref{mult:eq10}. First, 
\begin{align*}
& C\delta^{-6}  \, {(1+\alpha)^2 s}  \, { \log( \Xi / \delta)} \log(2N)  \left(\frac{4}{\pi}\right)^d \left[ K_d \left(\frac{2^{d-1}}{C}. \frac{ \varsigma {\log(1+\delta)}  }{ \log( \Xi / \delta)} \right) \right]^{-\frac12} \le m 
 \\
 \text{ if }\ \ \ & \frac {C}{m^2} \delta^{-12}  \, {(1+\alpha)^4 s^2}  \, { \log^2 ( \Xi / \delta)} \log^2 (2N)  \left(\frac{4}{\pi}\right)^{2d} \le K_d \left(\frac{2^{d-1}}{C}. \frac{ \varsigma {\log(1+\delta)}  }{ \log( \Xi / \delta)} \right). 
\end{align*}
Since $H_d$ is strictly increasing, the above inequality is equivalent to 
\begin{align*}
& H_d\left ( \frac {C}{m^2} \delta^{-12}  \, {(1+\alpha)^4 s^2}  \, { \log^2 ( \Xi / \delta)} \log^2 (2N)  \left(\frac{4}{\pi}\right)^{2d} \right) \le  \frac{2^{d-1}}{C}. \frac{ \varsigma {\log(1+\delta)}  }{ \log( \Xi / \delta)}. 
\end{align*}
From \eqref{mult:est16b} and Lemma \ref{lemmaM:est}, the right hand side is bounded from below by 
$$
\frac{\pi^d}{2^{3d} (d+1) } . \frac{  { \log(1+\delta)}  }{ C \log( \Xi / \delta)} .  \, \dfrac{\delta}{ (1+\alpha)^2 s} K_d \left(\dfrac{11\delta}{12 (1+\alpha)^2 s}.\dfrac{  \pi^{d} (d-1)! }{2^{4d}e(d+1)}\right), 
 $$
and we set $m$ such that 
\begin{align*}
& H_d\left( \dfrac{ 2^{3d} (d+1)} {  \pi^d }  \cdot \dfrac{ C (1+\alpha)^2 s { \log(\frac{ \Xi}{  \delta})} } {\delta^2} H_d\left ( \frac {C}{m^2}   . \frac{(1+\alpha)^4 s^2}{ \delta^{12}}  \, { \log^2 ( \frac{\Xi}{  \delta})} \log^2 (2N)  \left(\frac{4}{\pi}\right)^{2d} \right) \right)
\\
& \qquad \le \dfrac{11\delta}{12 (1+\alpha)^2 s}.\dfrac{  \pi^{d} (d-1)! }{2^{4d}e(d+1)}. 
\end{align*}   
Rearranging with notice that $H_d(\beta) \ge \beta,\forall \beta \le \frac{1}{e^{d-1}}$, it is sufficient that    
\begin{align*}
   m^2 \,  \ge  & \,  \frac{C { (d+1)^2} }{(d-1)!} \cdot  \frac{2^{11d}}{\pi^{4d}}  \cdot \frac{(1+\alpha)^8 s^4}{ \delta^{15}}  \, { \log^3 ( \frac{\Xi}{  \delta})} \log^2 (2N)  
    \\
    &\times    \left(\log \left( \frac {m^2}{C}   \, . \, \frac{ \delta^{12}}{(1+\alpha)^4 s^2}  \, . \, {\frac{1}{ \log^2 ( \frac{\Xi}{  \delta}) \log^2 (2N) }}  \left(\frac{\pi} {4}\right)^{2d} \right) \right)^{d-1}
    \\
    &\times \left(\log \left( \frac {m^2} {C { (d+1)} }  \, . \, \frac{ \delta^{14}} {(1+\alpha)^6 s^3}  \, . \,  \frac{1}{ \log^3 ( \frac{\Xi}{  \delta}) \log^2 (2N) }\, . \,  \frac{\pi^{3d}}{2^{7d}}   \right) \right)^{d-1}. 
\end{align*}
This condition can be replaced by the following shortened (yet slightly more restrictive) condition 
\begin{align*}
  m \,  & \ge   \,  \frac{C { (d+1)} }{{(d-1)!)^{\frac12}}} \cdot  \left( \frac{64\sqrt{2}}{\pi^2}\right)^d \cdot \frac{(1+\alpha)^4 s^2}{ \delta^{{15}/{2}}}  \, { \log^{3/2} ( \frac{\Xi}{  \delta})} \log (2N) 
 \left(\log \left( \! C{m}   \left(\! \frac{\pi}{4}\right)^{{d}}   \right) \right)^{{d-1}}. 
\end{align*}

Secondly, we need 
\begin{align*}
m \ge C  \, \frac{(1+\alpha)^2 s}{\delta^{6}}  \, { \log( \Xi / \delta)} \log(2N) \sqrt{\frac{m}{\gamma_0}}  \left(\frac{4}{\pi}\right)^{2d}  \left(\log\left(\frac{\pi^d}{4^{d}}\cdot \sqrt{\frac{m}{\gamma_0}}\right)\right)^{2d-2},  
\end{align*}
which is guaranteed given that 
\begin{align*}
m \ge \, \frac{C}{\gamma_0} \left(\frac{4}{\pi}\right)^{4d}   \frac{(1+\alpha)^4 s^2}{\delta^{12}}  \, { \log^2 ( \Xi / \delta)} \log^2(2N)    \left(\log\left(\frac{\pi^d}{4^{d}}\cdot \sqrt{\frac{m}{\gamma_0}}\right)\right)^{4d-4}. 
\end{align*}

The last condition is 
\begin{align*}
 m  \ge \,  C &  \delta^{-3} \pi^{-d}     {(1+\alpha)^2 {s}} \left [K_d \left(\!\dfrac{11\delta}{12 (1+\alpha)^2 s}.\dfrac{  \pi^{d} (d-1)! }{2^{4d}e(d+1)}\right) \right ]^{-1} \!\! { \log( \Xi / \delta)}    \log\left(\! \frac{ \log( \Xi / \delta)}{\gamma \delta}\!\right), 
\\
\text{ i.e., }\ & \dfrac{11\delta}{12 (1+\alpha)^2 s}.\dfrac{  \pi^{d} (d-1)! }{2^{4d}e(d+1)}  \ge  H_d\left(\frac{C}{m} . \frac{   {(1+\alpha)^2 {s}} }{ \delta^{3} \pi^{d}}  { \log( \Xi / \delta)}    \log\left( \frac{ \log( \Xi / \delta)}{\gamma \delta}\right) \right). 
\end{align*}
Rearranging and simplifying, we have 
\begin{align*}
m   \ge \dfrac{C (d+1)}  {  (d-1)! } \cdot \left( \frac{4 }{\pi }\right)^{2d} \cdot \frac{   {(1+\alpha)^4 {s}^2} }{ \delta^{4} }  { \log( \Xi / \delta)}    \log\left( \frac{ \log( \Xi / \delta)}{\gamma \delta}\right) \left(  \log\left(C{\pi^{d} m}     \right) \right)^{d-1} . 
\end{align*}

Finally, we observe from the definition of $\Xi$ and Lemma \ref{lemmaM:supp} that
\begin{gather}
\label{mult:last_est}
\begin{aligned}
\log( \Xi / \delta) & = \frac12 \log\( \dfrac{(1+\alpha)^2 {s}}{\delta^2 \pi^{d} {e} }\) - \frac12\log\(  K_d\left(\dfrac{11\delta}{12 (1+\alpha)^2 s}.\dfrac{  \pi^{d} (d-1)! }{2^{4d}e(d+1)}\right) \)
\\
& \le  \frac12 \log\( \dfrac{(1+\alpha)^2 {s}}{\delta^2 \pi^{d} {e} }\) +  \log\(\frac{C  (d+1)(d-1)^{d-1}  }{  (d-1)!    }\cdot \dfrac{2^{4d}}{\pi^{d} }\cdot \dfrac{(1+\alpha)^2 s}{\delta}  \)
\\
& \le  \log\(\frac{C  (d+1)(d-1)^{d-1}  }{  (d-1)!    }\cdot \dfrac{2^{4d}}{\pi^{{3d}/{2}} }  \cdot \dfrac{(1+\alpha)^3 s^{\frac32}}{\delta^{2}}  \).
\end{aligned}
\end{gather}
Subsuming \eqref{mult:last_est} into the three above conditions on $m$, we deduce that \eqref{mult:complexity} guarantees 
$$
\|\bA \bz\|_2^2 > (1 - 13\delta)\|\bz\|_{2}^2, \ \ \forall \bm{z}\in \cE_{s,\alpha},
$$
with probability exceeding $1-(\gamma + \gamma_0)$. Rescaling $\delta$ concludes the proof. 
$\square$

\section{Numerical illustrations}
\label{sec:num_examples}
In this section, we present some numerical experiments to illustrate an important observation from our theory, that is, restricting the sparse recovery on random Legendre matrices associated with preferable sets (in particular, sample sets with small test values) can improve the reconstruction performance. 

These tests are conducted for $1d$ Legendre approximations, where we set $\cJ := \{1,\ldots, 360\}$. For each experiment, we generate a large number of sets of samples (drawn independently from the uniform distribution in $[-1,1]$) and rank them according to their test values $ \sum_{i=1}^m Z_i$, where $\gamma_0$ is set to $0.8$. We divide the sample sets into five equal-sized groups: the first group contains $20\%$ of the sets with lowest test values (roughly, preferable sets according to 0.8), the second group contains next $20\%$ of the sets with lowest test values and not being in the first group, and so on. The fifth group includes $20\%$ of the sets with highest test values. 

In the first experiment, we generate $1000$ sample sets, each of which has 180 samples; hence, random matrices have fixed size $180\times 360$. For each matrix, we form a sparse coefficient vector $\bc$ by randomly selecting the support of $\bc$ from $\cJ$ and drawing its coefficients from a Gaussian distribution, and aim to reconstruct $\bc$ from noiseless observations $\bg = \bA\bc$. The average errors and successful rates associated with five groups of sample sets are plotted in Figure \ref{fig:num1} for two different sparsity levels of $\bc$ ($30$ and $40$). We observe that on average, reconstruction on sample sets with smaller test values results in reduced recovery errors and better successful rates, thus confirming our theory.   

\begin{center}
\begin{figure}[h]
\centering
\includegraphics[height=4.6cm]{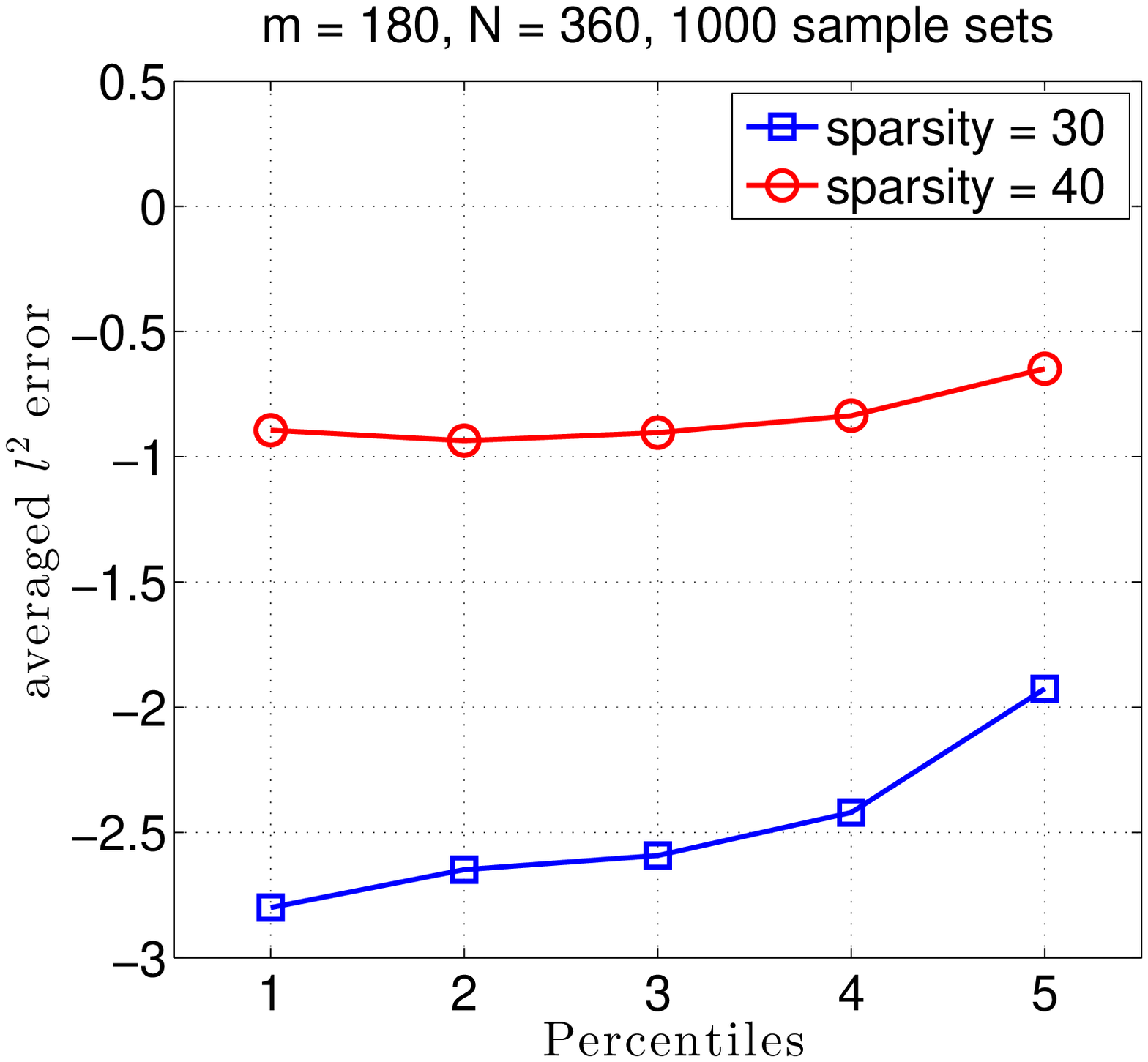} 
\includegraphics[height=4.6cm]{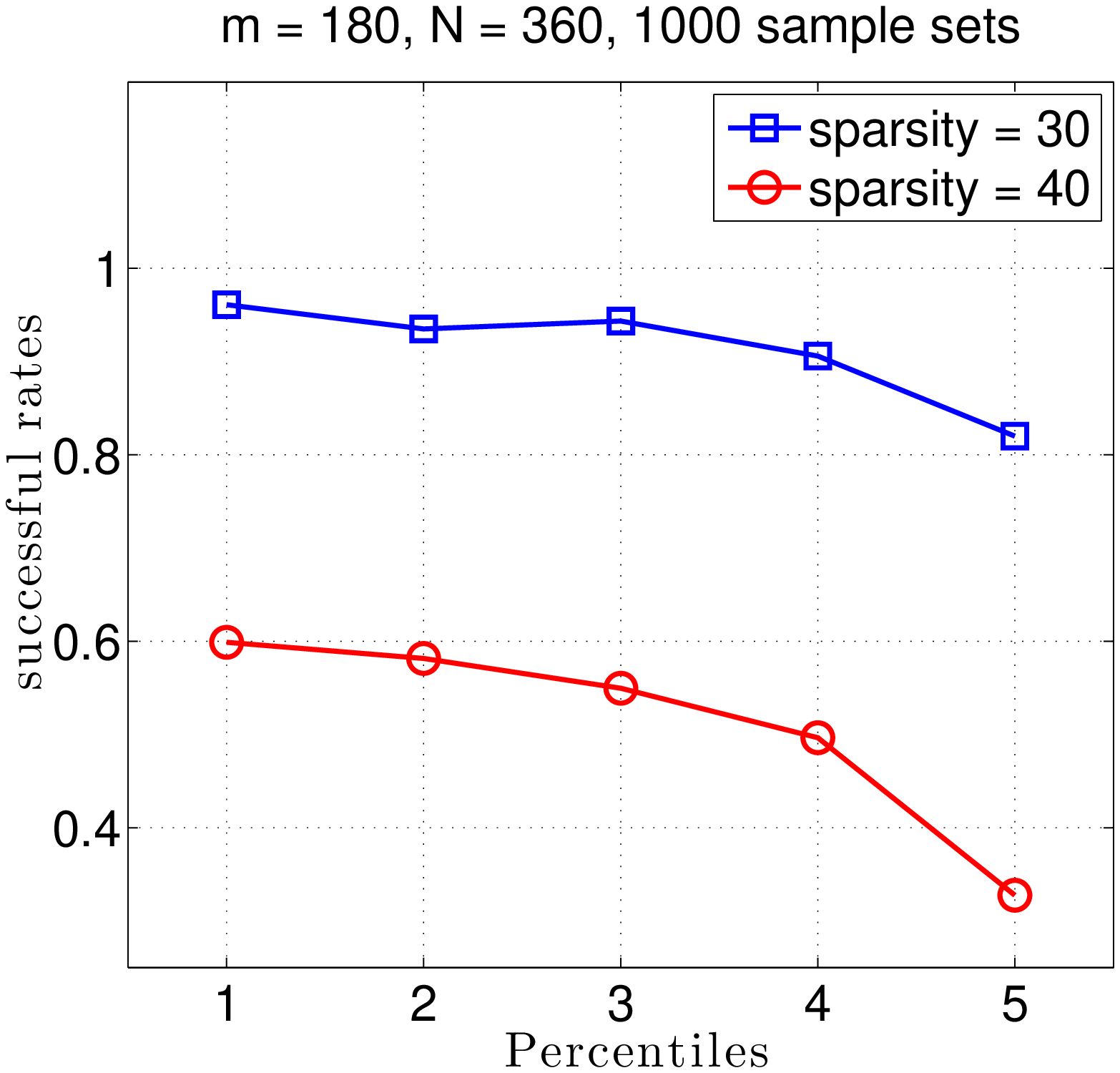} 
\caption{Averaged recovery errors and successful rates associated with five different groups of sample sets with increasing test values in the reconstruction of $30$- and $40$-sparse coefficient vectors using random Legendre matrices. }
\label{fig:num1}
\end{figure}
\end{center}

In the second experiment, we generate $1500$ sample sets and again seek to reconstruct sparse vector $\bc$ from noiseless observations $\bg = \bA\bc$. The sparsity of $\bc$ is fixed at $14$, and the number of samples varies from $10$ to $180$. In Figure \ref{fig:num2}, we show the average errors and successful rates associated with $20\%$ of sample sets with lowest test values (roughly, preferable sets according to $0.8$) and $20\%$ of sets with highest test values. For comparison, we also plot the average errors and successful rates associated with all $1500$ sample sets, as well as those from preconditioned Legendre matrices with Chebyshev sampling, \cite{RauWard12,HD15,JakemanNarayanZhou16}. This technique is widely-accepted to be the optimal sampling strategy to reconstruct coefficient vectors from underdetermined Legendre matrices. To produce the results for preconditioning technique, $1500$ sets of samples drawn from Chebyshev distribution in $[-1,1]$ are generated separately. Figure \ref{fig:num2} reveals that preferable sets according to $0.8$ are superior to general uniform sample sets and sets with high test values in sparse recovery. Although not outperforming Chebyshev sampling, preferable sets provide a simple selective strategy to improve the recovery property with uniform sampling.

\begin{center}
\begin{figure}[h]
\centering
\includegraphics[height=4.5cm]{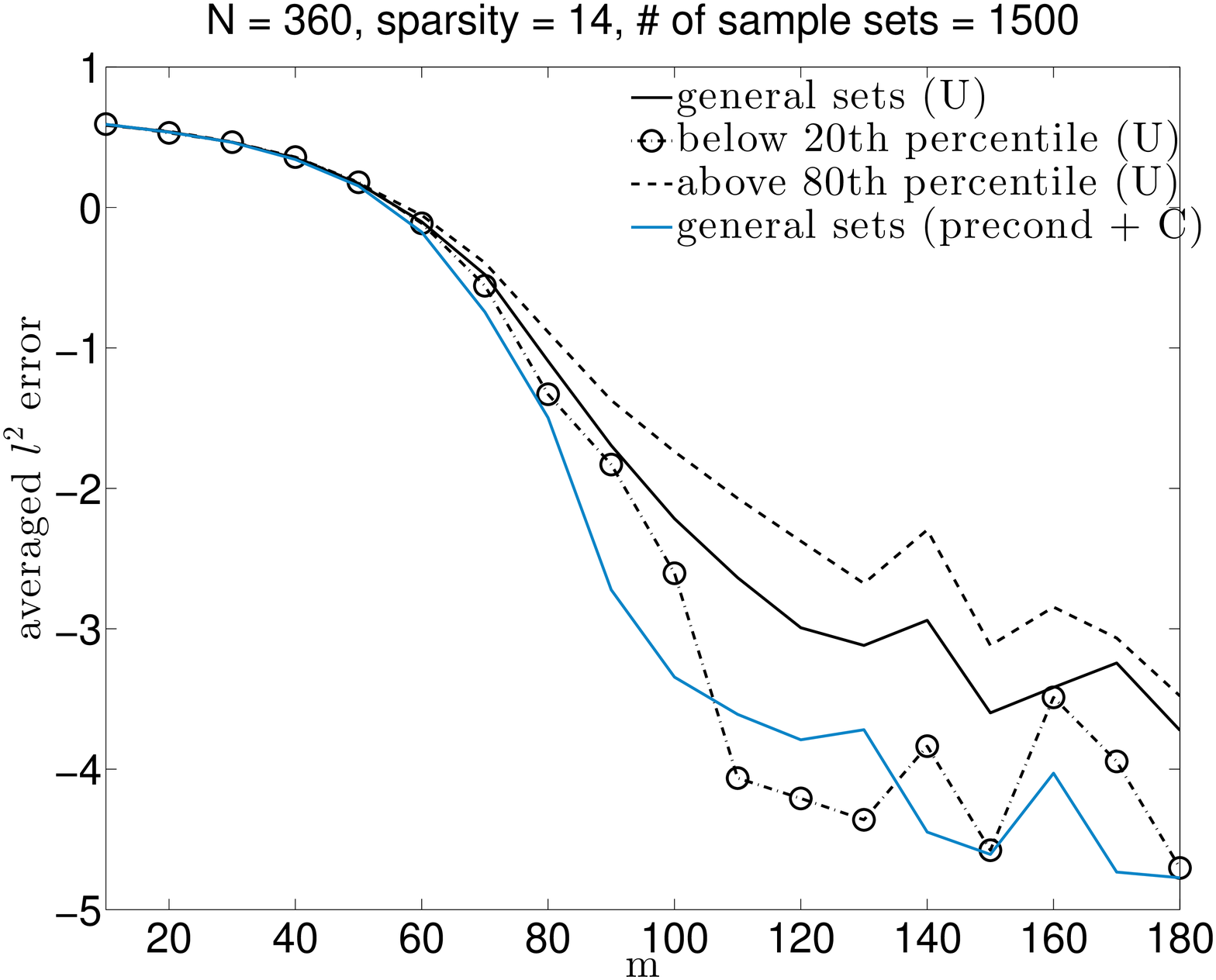} 
\includegraphics[height=4.5cm]{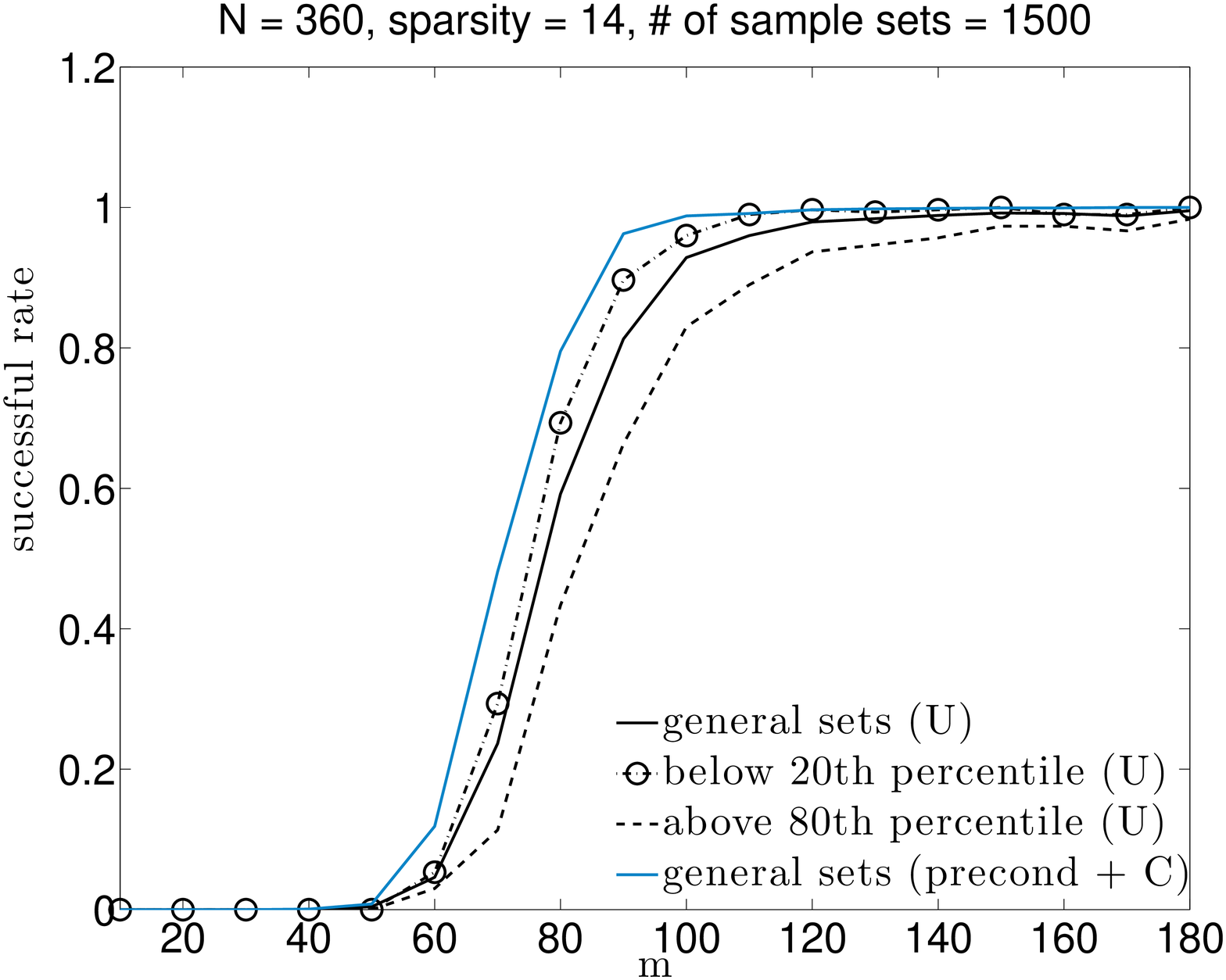} 
\caption{Averaged recovery errors and successful rates associated with $20\%$ of sample sets with lowest test values and $20\%$ of sample sets with highest test values in the reconstruction of $14$-sparse coefficient vectors using random Legendre matrices. The recovery performances over all uniform sample sets \textit{(solid black)}, as well as using preconditioned Legendre matrices with Chebyshev sampling \textit{(solid blue)} are shown for comparison.}
\label{fig:num2}
\end{figure}
\end{center}

The third test is similar to our second one; however, we fix the number of samples at $180$ and plot the average errors and successful rates with increasing sparsity. Again, sample sets with low test values show much better reconstruction, compared to general uniform sample sets and sets with high test values (Figure \ref{fig:num3}).  

\begin{center}
\begin{figure}[h]
\centering
\includegraphics[height=4.5cm]{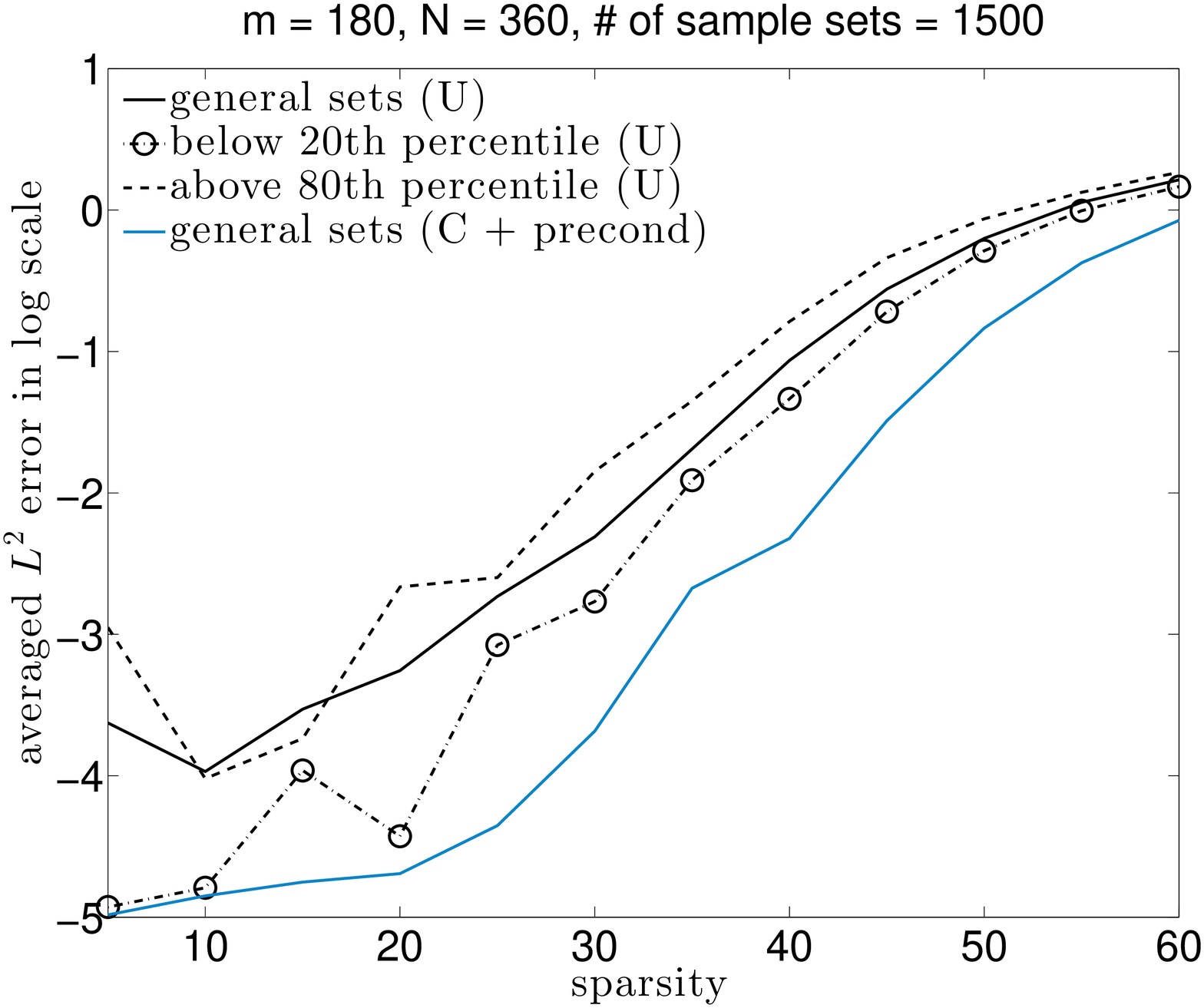}
\includegraphics[height=4.5cm]{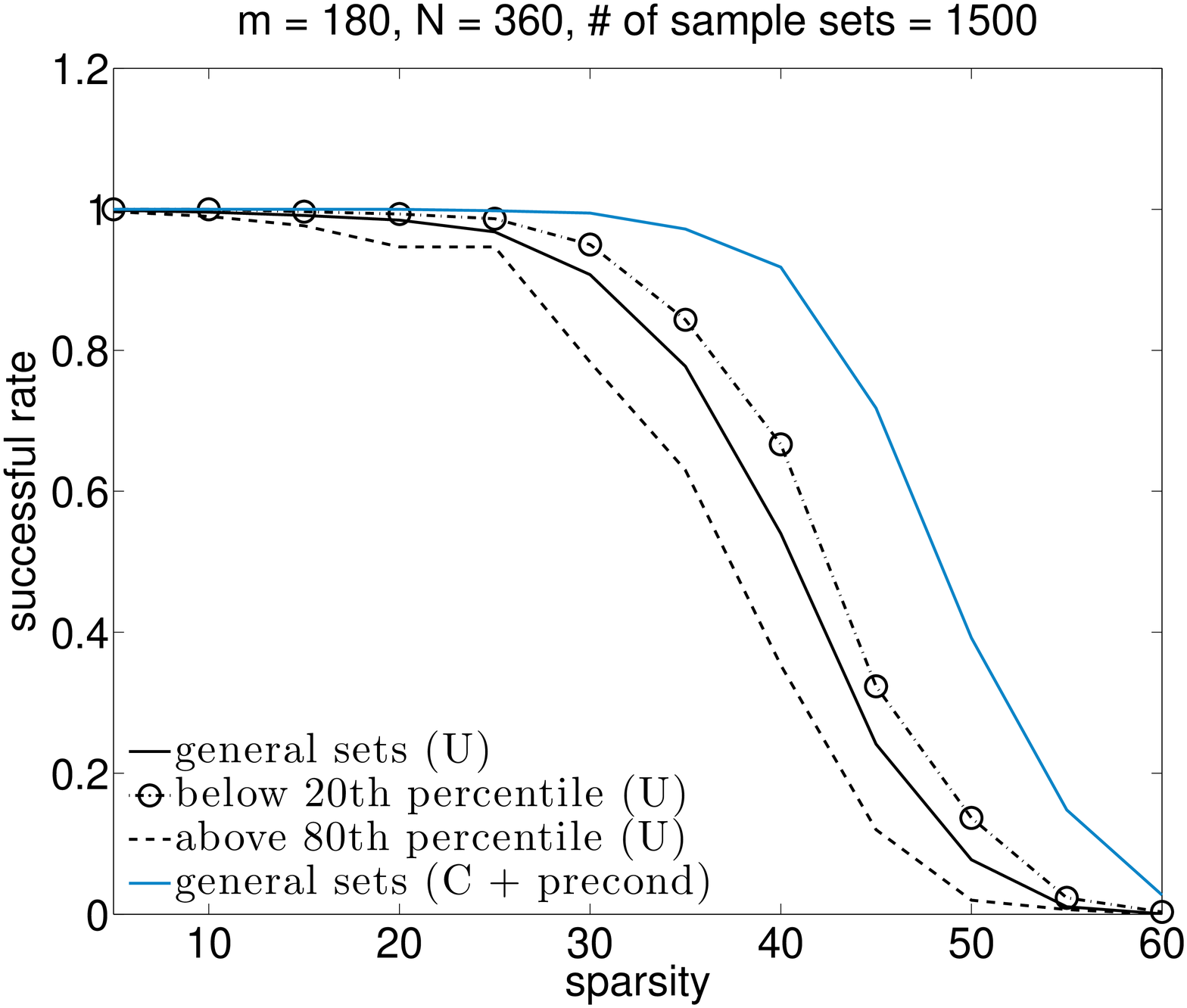} 
\caption{Averaged recovery errors and successful rates associated with $20\%$ of sample sets with lowest test values and $20\%$ of sample sets with highest test values in the reconstruction of sparse coefficient vectors using random Legendre matrices of size $180\times 360$. The recovery performances over all uniform sample sets \textit{(solid black)}, as well as using preconditioned Legendre matrices with Chebyshev sampling \textit{(solid blue)} are shown for comparison.}
\label{fig:num3}
\end{figure}
\end{center}

\section{Concluding remarks}
\label{sec:conclusion}
{
This paper provides a theoretical justification for the sparse reconstruction from underdetermined Legendre systems with uniform samples via $\ell_1$ minimization. Our analysis employs the envelop bound (rather than the prohibitive uniform bounds) of all Legendre polynomials, and by extending recent chaining arguments [4, 8] to deal with this bound, allows us to establish a new uniform recovery guarantee for sparse, multi-dimensional Legendre expansions, which is independent of the polynomial subspaces. To the best of our knowledge, this is the first time recovery condition is established for orthonormal systems without assuming the uniform boundedness of the sampling matrix. Extending the present results to related scenarios, such as non-uniform recovery and discrete least squares, to relax the dependence of sample complexity on Legendre uniform bound would be the next logical step. Also, we believe that the analysis approach herein should be of interest on its own, and can be applied elsewhere, to other random systems which can be bounded precisely pointwise, but whose uniform bound is bad.
}
\appendix
\label{sec:appendix}

\bibliographystyle{spmpsci}      
\bibliography{database3}   
\end{document}